\documentclass[a4paper, 11pt]{article}

\usepackage{amsfonts}
\usepackage{amsmath}
\usepackage{amssymb}
\usepackage{mathrsfs}
\usepackage{amsthm}
\usepackage{color}

\usepackage{vmargin}

\setmarginsrb{20mm}{20mm}{20mm}{20mm}{10mm}{10mm}{10mm}{10mm}

\usepackage[T1]{fontenc} %skandit

\theoremstyle{plain}
\newtheorem{lem}{\textsc{Lemma}}
\newtheorem{thm}{\textsc{Theorem}}
\newtheorem{prop}{\textsc{Proposition}}
\newtheorem{cor}{\textsc{Corollary}}
\newtheorem{defn}[equation]{\textsc{Definition}}

\theoremstyle{remark}
\newtheorem{rem}{{\bf \textsc{Remark}}}
\newtheorem{ex}{{\bf \textsc{Example}}}

\newcommand{\R}{\mathbb{R}}
\newcommand{\N}{\mathbb{N}}
\newcommand{\Rn}{\mathbb{R}^n}
\newcommand{\Om}{\Omega}
\def\px{p(x)}
\def\p{{p(\cdot)}}
\def\pstarx{p^{*}(x)}
\def\pstar{p^{*}(\cdot)}
\def\pstarm{p^{*}_{-}}
\def\pstarp{p^{*}_{+}}
\def\rx{r(x)}
\def\rdot{{r(\cdot)}}
\def\sigmadot{{\sigma(\cdot)}}
\def\tsigmax{\tilde{\sigma}(x)}
\def\tdeltax{\tilde{\delta}(x)}
\def\ttaux{\tilde{\tau}(x)}

\newcommand{\Wpx}{W^{1,p(\cdot)}}
\newcommand{\Wpxzero}{W^{1,p(\cdot)}_{0}}
\newcommand{\Wpxloc}{W^{1,p(\cdot)}_{loc}}

\def\osc{\operatornamewithlimits{osc}}
\def\essinf{\operatornamewithlimits{ess\,inf}}

\def\supp{\operatornamewithlimits{supp}}
\DeclareMathOperator*{\esssup}{ess\,sup}
\def\Xint#1{\mathchoice
   {\XXint\displaystyle\textstyle{#1}}%
   {\XXint\textstyle\scriptstyle{#1}}%
   {\XXint\scriptstyle\scriptscriptstyle{#1}}%
   {\XXint\scriptscriptstyle\scriptscriptstyle{#1}}%
   \!\int}
\def\XXint#1#2#3{{\setbox0=\hbox{$#1{#2#3}{\int}$}
   \vcenter{\hbox{$#2#3$}}\kern-.5\wd0}}
\def\vint{\Xint-}

\DeclareMathOperator{\divop}{div}
\renewcommand{\div}{\divop}
\renewcommand{\phi}{\varphi}

\newcommand{\calF}{{\mathcal F}}
\newcommand{\calFone}{{\mathcal {F}_1}}
\newcommand{\calFtwo}{{\mathcal {F}_2}}

\date{\today}

\newcommand{\Chiado}{Chiad\`o\ }

\newcommand{\kom}[1]{}
\renewcommand{\kom}[1]{{\bf [#1]}}

\definecolor{blau}{rgb}{0.1,0.0,0.9}

\newcounter{komcounter}
\numberwithin{komcounter}{section}

\usepackage{comment}

\title{\bf H\"older continuity of quasiminimizers with nonstandard growth}
\author{
Tomasz Adamowicz{\small{$^1$}}
\\
\it\small Institute of Mathematics of the Polish Academy of Sciences \\
\it\small Sniadeckich 8, 00-656 Warsaw, Poland\/{\rm ;}
\it\small T.Adamowicz@impan.pl
\\
\\
Olli Toivanen{\small{$^2$}}
\\
\it\small Institute of Mathematics of the Polish Academy of Sciences \\
\it\small Sniadeckich 8, 00-656 Warsaw, Poland\/{\rm ;}
\it\small O.Toivanen@impan.pl
}

\date{}

\begin{document}

\maketitle

\footnotetext[1]{T. Adamowicz was supported by a grant of National Science Center, Poland (NCN),
UMO-2013/09/D/ST1/03681.}
\footnotetext[2]{O. Toivanen was supported by a grant from the Warsaw Center of Mathematics and Computer Science.}

\begin{abstract}
   We show the H\"older continuity of quasiminimizers of the energy functionals $\int f(x,u,\nabla u)\,dx$ with nonstandard growth under the general structure conditions
\[
|z|^{p(x)} - b(x)|y|^{r(x)}-g(x) \leq f(x,y,z) \leq \mu|z|^{p(x)} + b(x)|y|^{r(x)} + g(x).
\]
 The result is illustrated by showing that weak solutions to a class of $(A,B)$-harmonic equations
 \[
  -\div A(x,u,\nabla u) = B(x,u,\nabla u),
 \]
  are quasiminimizers of the variational integral of the above type and, thus, are H\"older continuous. Our results extend work by Chiad\`o Piat--Coscia~\cite{Chiado-Coscia}, Fan--Zhao~\cite{Fan-Z} and Giusti--Giaquinta~\cite{Gia-Giu}.
\end{abstract}
\bigskip

\medskip

%\komT{Perhaps 35J20 should be primary?}

\section{Introduction}

 The regularity theory for minima and quasiminima of energy functionals with nonstandard growth and related PDEs have been a subject of intensive studies in the last two decades. The importance of such studies grows largely from the applications, for instance in electro-rheological fluids, Acerbi--Mingione~\cite{am02}, in fluid dynamics, Diening--R\r u\v zi\v cka~\cite{dr}, in the study of image processing, Chen--Levine--Rao~\cite{clr} and in the model of thermistor, Zhikov~\cite{z97}; see Harjulehto--H\"ast\"o--L\^e--Nuortio~\cite{hhn} for a recent survey and further references, see also the monograph by R\r u\v zi\v cka~\cite{Ru00} and R\v adulescu~\cite{Rad}. The simplest energy functional studied in the variable exponent setting is the one associated with the $\p$-harmonic operator:
 \[
  \int_{\Om} |\nabla u(x)|^{\px}dx
 \]
for a measurable function $p:[1,\infty)\to \R$, domain $\Om\subset \R^n$ and a function $u\in \Wpxloc(\Om)$. If $p=const$, then this energy reduces to the classical $p$-harmonic energy with the related $p$-harmonic equation arising as the Euler-Lagrange equation for the $p$-Dirichlet energy. In spite of the symbolic similarity to the constant exponent case, unexpected phenomena occur already in the case $n=1$, as it turns out that the minimum of the $\p$-Dirichlet energy may fail to exist even for smooth functions $p$, see Examples 3.2 and 3.6 in \cite{hhn}.

 The main purpose of this paper is to study the energy functionals
\[
\mathcal{F}_{\Om}(u)=\int_\Omega f(x,u,\nabla u)\,dx
\]
for $f:\Om\times \R\times \R^n\to\R$, under the general growth conditions
\[
|z|^{p(x)} - b(x)|y|^{r(x)}-g(x) \leq f(x,y,z) \leq \mu|z|^{p(x)} + b(x)|y|^{r(x)} + g(x),
\]
 where $b,g$ are nonnegative functions in the appropriate variable exponent Lebesgue spaces, see \eqref{F-growth} and \eqref{F-growth2} for details and Preliminaries for further introduction to the topic. Under such growth conditions, Toivanen~\cite{Toivanen-2} showed the local boundedness of minimizers. Here we continue and extend further investigations to include the case of quasiminimizers, see Definition~\ref{quasi-def}, and to show their H\"older continuity, see Theorem~\ref{new-Theorem-4-1}. The novelty of our results lies largely in the fact that we allow coefficients $b$ and $g$ to be integrable functions. Especially the fact that $b=b(x)$ requires extra attention comparing to results for $b=const$ or $b\in L^{\infty}$. In a consequence one needs to extend some of the de Giorgi estimates to our general setting, see Lemma~\ref{Caccioppoli} for the Caccioppoli-type estimates for quasiminimizers of $\mathcal{F}_{\Om}$. Moreover, we provide fine $L^\infty$ estimates involving level sets (Theorem~\ref{new-prop-4.2}).

For the history of the problem, we note that the H\"older continuity of quasiminimizers with nonstandard growth has been proved by Chiad\`o Piat--Coscia~\cite[Theorem 4.1]{Chiado-Coscia} and Fan-Zhao~\cite[Theorem 3.1]{FZ} under the assumption that $b=g=const$, cf. (1.3) in \cite{Chiado-Coscia} and (3.1) in \cite{FZ}. These results were subsequently generalized to the case $b=const$, $g\in L^{s(\cdot)}$ for $s(\cdot)>n/p(\cdot)$ in Fan-Zhao~\cite[Theorem 3.2]{Fan-Z} and applied in the studies of PDEs with nonstandard growth, see Theorem 2.2 in \cite{Fan-Z}. For further results on the H\"older regularity for minimizers we refer to Acerbi--Mingione~\cite{am01}, in the case $b\equiv g\equiv 0$, Eleuteri-Habermann~\cite{EH}, for the obstacle problem, see also Eleuteri~\cite{ele} and Mingione~\cite{min} for a survey on regularity of minima. We would like also to add that the related weak Harnack inequality still remains an open problem for quasiminimizers under the general framework studied in our paper (see Harjulehto--Kuusi--Lukkari--Marola--Parviainen~\cite{HKLMP} for the case $b\equiv g\equiv 0$).

  Applications of Theorem~\ref{new-Theorem-4-1} are studied in Section~\ref{Section4}. There, we show the H\"older continuity of solutions to a class of PDEs with nonstandard growth of type
  \[
  -\div A(x,u,\nabla u) = B(x,u,\nabla u),
  \]
  under growth assumptions on $A$ and $B$, see \eqref{AB-cond} and \eqref{AB-exp-cond} for details. Namely,
  we prove that solutions are quasiminimizers of $\mathcal{F}_{\Om}$ with $b$ and $g$ defined in terms of growth parameters of $A$ and $B$, see Theorem~\ref{AB-quasim}. The fact that, under our growth assumptions, the coefficient $b$ may depend on $x\in \Om$ allows us to cover wider classes classes of PDEs than those studied so far in the literature. We illustrate our discussion with Examples 1 and 2, also cf. Theorem 2.2 in Fan--Zhao~\cite{Fan-Z}.

  Finally, in Appendix we prove Lemmas~\ref{new-lem-2.4} and \ref{new-ren-2.5} formulated in Preliminaries and needed to show Theorems~\ref{new-prop-4.2} and \ref{new-Theorem-4-1}. To our best knowledge, the proofs of those lemmas are not available in the literature for energy functionals under the general growth conditions.
\section{Preliminaries}

A measurable function $p\colon \Om\to [1,\infty]$ is called a \emph{variable exponent}. Let $A\subset\Omega$. We say that $\p$ is a bounded exponent in $A$ if it holds that:
\begin{equation*}
   1 < p^{-}_{A} \le p^{+}_{A} < \infty, \quad \text{where }\quad p^{-}_{A}=\essinf_{A} p \quad
     \text{ and }\quad p^{+}_{A}=\esssup_{A} p.
\end{equation*}
If $A=\Om$ or if the underlying domain is fixed, we will often skip the index and set $p_A=p_\Om=p$.

We say that $p$ is log-H\"older continuous, if there is a constant
$L>0$ such that
\begin{equation*}  \label{eq-def-log-Holder}
|p(x)-p(y)|\leq \frac{L}{\log(e+1/|x-y|)}
 \quad \text{for } x,y\in \Om.
\end{equation*}

The log-H\"older continuity condition plays an important role in the theory of variable exponents, for instance in the results on the boundedness of the maximal Hardy-type operators and in the studies of density of smooth functions in the variable exponent Sobolev spaces, see e.g. Chapters 4 and 9 in Diening--Harjulehto--H{\"a}st{\"o}--R{\r u}{\v z}i{\v c}ka~\cite{DHHR}.

In what follows we will assume that a variable exponent $\p$ satisfies one of the following assumptions:
\begin{align}\label{p-assump}
&\hbox{(1) } \p \hbox{ is a bounded Lipschitz continuous,} \nonumber \\
& \phantom{AAAAAAAA} \hbox{ or } \\
&\hbox{(2) } \p \hbox{ is a bounded log-H\"older continuous and } \|\nabla p\|_{L^{s}(\Om)}<\infty \hbox{ for some } s>n. \nonumber
\end{align}

For background on variable exponent function spaces we refer to the
monograph \cite{DHHR}.

We define a \emph{(semi)modular} on the set of measurable functions
by setting
\[
\varrho_{L^{p(\cdot)}(\Omega)}(u) :=\int_{\Omega} |u(x)|^{p(x)}\,dx;
\]
here we use the convention $t^\infty = \infty \chi_{(1,\infty]}(t)$
in order to get a left-continuous modular, see
\cite[Chapter~2]{DHHR} for details. The \emph{variable
exponent Lebesgue space $L^{p(\cdot)}(\Omega)$} consists of all
measurable functions $u\colon \Omega\to \R$ for which the modular
$\varrho_{L^{p(\cdot)}(\Omega)}(u/\mu)$ is finite for some $\mu >
0$. The \emph{Luxemburg norm} on this space is defined as
\[
\|u\|_{L^{p(\cdot)}(\Omega)}:= \inf\Big\{\mu > 0\,\colon\,
\varrho_{L^{p(\cdot)}(\Omega)}\big(\tfrac{u}\mu\big)\leq 1\Big\}.
\]
Equipped with this norm, $L^{p(\cdot)}(\Omega)$ is a Banach space.
The variable exponent Lebesgue space is a special case of the
Musielak-Orlicz space, cf. Kov{\'a}{\v c}ik--R{\'a}kosn{\'\i}k~\cite{KR} and Cruz-Uribe--Fiorenza~\cite{CruzUF}. For a constant function $p$ it coincides with the classical Lebesgue space.

There is no functional relationship between norm and modular, but we
do have the following useful inequality, the so-called \emph{unit ball property}, see Lemmas 3.2.4 and 3.2.5 in \cite{DHHR}:
\begin{equation}\label{unit-ball-prop}
\min\Big\{ \varrho_{L^{p(\cdot)}(\Omega)}(u)^\frac1{p^-}, \varrho_{L^{p(\cdot)}(\Omega)}(u)^\frac1{p^+} \Big\}
\le
\| u\|_{L^{p(\cdot)}(\Omega)}
\le
\max\Big\{ \varrho_{L^{p(\cdot)}(\Omega)}(u)^\frac1{p^-}, \varrho_{L^{p(\cdot)}(\Omega)}(u)^\frac1{p^+} \Big\}.
\end{equation}

If $\Om$ is a measurable set of finite measure and $p$ and $q$ are
variable exponents satisfying $q\leq p$, then $L^{p(\cdot)}(\Om)$
embeds continuously into $L^{q(\cdot)}(\Om)$. In particular, every
function $u\in L^{p(\cdot)}(\Omega)$ also belongs to
$L^{p_\Omega^-}(\Omega)$. The \emph{variable exponent H\"older inequality}
takes the form
\begin{equation*}\label{Hold-ineq}
\int_\Omega u v \,dx \le 2 \, \|u\|_{L^{p(\cdot)}(\Omega)} \|v\|_{L^{p'(\cdot)}(\Omega)},
\end{equation*}
where $p'$ is the point-wise \textit{conjugate exponent}, $1/p(x)
+1/p'(x)\equiv 1$.

In what follows we will frequently appeal to the pointwise \emph{Young inequality} and its parameter variant for a given $\epsilon \in(0, 1]$. For the sake of completeness of discussion let us formulate this inequality:
\begin{equation*}\label{young-ineq}
 \int_{\Om} u(x)v(x) \,dx \leq \int_{\Om}\tfrac{\epsilon^{\px}}{\px} u(x)^{\px}\,dx +\int_{\Om}\tfrac{\epsilon^{-p'(x)}}{p'(x)} v(x)^{p'(x)}\,dx.
\end{equation*}

The \emph{variable exponent Sobolev space $W^{1,p(\cdot)}(\Omega)$}
consists of functions $u\in L^{p(\cdot)}(\Omega)$ whose
distributional gradient $\nabla u$ belongs to
$L^{p(\cdot)}(\Omega)$. The variable exponent Sobolev space
$W^{1,p(\cdot)}(\Omega)$ is a Banach space with the norm
\begin{displaymath}
\|u\|_{L^{p(\cdot)}(\Omega)}+\|\nabla u\|_{L^{p(\cdot)}(\Omega)}.
\end{displaymath}
In general, smooth functions are not dense in the variable exponent
Sobolev space \cite[Section~9.2]{DHHR}, but the log-H\"older condition suffices
to guarantee that they are \cite[Section~8.1]{DHHR}.
In this case, we define \emph{the Sobolev space with zero
boundary values}, $W_0^{1,\p}(\Omega)$, as the closure of $C_0^\infty(\Omega)$
in $W^{1,\p}(\Omega)$. The local Sobolev space $\Wpxloc(\Om)$ is defined in a similar way as in the constant exponent case.

Another fundamental tool employed in the paper is the concept of a \emph{level set}, see e.g. a book by Giusti~\cite{Giusti}. Denote by $B_R$ a ball in $\mathbb{R}^n$ with radius $R>0$. If $u:\Om\to\R$ and $k\in\R$, then we set
\[
 A(k, R):=\{x\in \Om: u(x)>k\}\cap B_R.
\]
Furthermore, we recall notation
\[
 M(u,R)=\sup_{B_R}u,\qquad  m(u,R)=\inf_{B_R}u,\qquad \osc(u, R)=M(u,R)-m(u,R).
\]
By $f_A$ we denote the integral average of function $f$ over set $A$, that is
\[
 f_A:=\vint_{A} f dx=\frac{1}{|A|}\int_{A} f dx.
\]
For the sake of convenience of notation and in order to simplify the presentation (especially in the proof of Theorem~\ref{new-prop-4.2}), we will slightly abuse the above notation for the average value over the level sets $A(k,R)$ and denote
\[
 \vint_{A(k, R)} f dx=\frac{1}{|B_R|}\int_{A(k, R)} f dx.
\]

We will now recall basic definitions for quasiminimizers in the variable exponent setting.

Let $\Om\subset \R^n$ be a bounded domain and let $f=f(x, t, p):\Om\times\R\times\Rn\to \R$ for $n\geq 1$ be a Carath\'eodory function, i.e. measurable as a function of $x$ for every $(t,p)$ and continuous in $(t,p)$ for almost every $x\in \Om$. Define an energy functional
$$F_{\Om}(u)=\int_{\Om}f(x, u(x), \nabla u(x))dx.$$
\begin{defn}\label{quasi-def}
 Let $K\geq 1$. We say that a function $u\in \Wpxloc(\Om)$ is a \emph{$K$-quasiminimizer of $F_\Om$} if for every open $\Om'\Subset \Om$ and for all $v\in \Wpxloc(\Om)$ such that $u-v\in \Wpxzero(\Om')$ we have
 \[
  F_{\Om'}(u)\leq K F_{\Om'}(v).
 \]
 Equivalently, $u$ is $K$-quasiminimizer if for every $\phi\in \Wpxloc(\Om)$ with $\supp \phi\Subset \Om$ we have
 \[
  F_{\supp \phi}(u)\leq K F_{\supp \phi}(u+\phi).
 \]
\end{defn}
If $K=1$, then we retrieve the definition of local minimizers, i.e. a local minimum of $F_\Om$ is a $1$-quasiminimum.

Following the discussion in Fan-Zhao, see \cite[Definition 2.2]{Fan-Z}, we say that a function $u\in \Wpxloc(\Om)$ is a \emph{local quasiminimizer of $F$} if for every $x\in \Om$ there exists a neighborhood $\Om_x\subset \Om$ of $x$ such that $u$ is a quasiminimizer of $F$ in $\Om_x$. From the point of view of applications of our results on local H\"older continuity of quasiminimizers it will often suffice to show that a function is a local quasiminimizer, cf. Section~\ref{Section4}.

Quasiminimizers have been studied in various settings and contexts: in the Euclidean setting, see e.g. Giaquinta--Giusti~\cite{Gia-Giu} and Giusti~\cite{Giusti}, in the setting of metric spaces, see e.g. Kinnunen--Martio~\cite{KinMar}, in the variable exponent setting, see e.g. Harjulehto--Kuusi--Lukkari--Marola--Parviainen~\cite{HKLMP}. One also studies relations between quasiminimizers and elliptic equations, see e.g.~\cite{Gia-Giu}, Martio~\cite{Mar1} and quasiminimizers in  the parabolic setting, see e.g. Kinnunen--Masson~\cite{KinMas} and references therein.

In \cite{Toivanen-2}, Toivanen showed the local boundedness of local minimizers of
\begin{equation}\label{energy-fun}
\mathcal{F}_{\Om}(u)=\int_\Omega f(x,u,\nabla u)\,dx,
\end{equation}
where $f:\Om\times \R\times \R^n\to \R$ is subject to the general structural conditions
\begin{equation}\label{F-growth}
|z|^{p(x)} - b(x)|y|^{r(x)}-g(x) \leq f(x,y,z) \leq \mu|z|^{p(x)} + b(x)|y|^{r(x)} + g(x).
\end{equation}
As for coefficients, we a priori assume the following:
\begin{align}\label{F-growth2}
  & \mu\geq 1 \hbox{ is a constant, }\quad p(x)\leq r(x)\leq p^{*}(x) \hbox{ for } x\in \Om, \nonumber \\
  & b\geq 0 \hbox{ and } b\in L^{\sigmadot}(\Om) \hbox{ for }\sigma \in C^{0}(\Om) \hbox{ such that } \sigma(x)>\frac{p^{*}(x)}{p^{*}(x)-r(x)} \hbox{ for } x \in \Om, \\
  & g\geq 0 \hbox{ and } g\in L^{t}(\Om), t>\frac{n}{p^-}. \nonumber
\end{align}
Here $p^*(\cdot) = \frac{p(\cdot)n}{n - p(\cdot)}$ is the Sobolev conjugate exponent for $p < n$. For $p > n$ functions in $W^{1,\p}(\Om)$ are H\"older continuous by the Sobolev embedding theorem. Thus we will limit our discussion only to the case $p < n$.

Let us now define an auxiliary energy functional. It will turn out that in several cases our discussion of properties of the energy $\mathcal{F}_{\Om}$ can be reduced to the analysis of the analogous properties of $\calFone_{, \Om}$:
\begin{equation}\label{def-F1}
\calFone_{, \Om}(u):=\int_{\Om}(\mu|\nabla u|^{\px}+b(x)|u|^{\rx}+g(x)+h(x))\,dx,
\end{equation}
where $\mu$, $p, r$, $b, g$ are the same as in the definition of $\calF_\Om$, whereas $h(x):=c\,b(x)^{\frac{\pstarx}{\pstarx-\rx}}$ for $x\in \Om$ with constant $c$ depending among other parameters on $K, p, r$, $\|b\|_{L^{\sigmadot}(\Om)}, \|g\|_{L^t(\Om)}$ and $\|u\|_{\Wpxloc(\Om)}$. In what follows we will omit the symbol of the domain and write $\calF(u)$ and $\calFone(u)$ if the domain is fixed or clear from the context of discussion.

\begin{lem}\label{new-lem-2.4}
Every bounded $K$-quasiminimizer of $\mathcal{F}_{\Om}$ is a $K'$-quasiminimizer of $\calFone_{, \Om'}$ for any $\Om'\Subset \Om$, where $K'=2(K+c)$ with $c$ depending on $p^+, p^-, n, r^+, \mu$, norms of coefficients $\|b\|_{L^{\sigma(\cdot)}},\|g\|_{L^{t}}$, also on $\|u\|_{\Wpxloc}$.
\end{lem}
To our best knowledge, the proof of this lemma is not available in the literature for such general assumptions, i.e. for $b=b(x)$ with $b\in L^{\sigmadot}$ and $g\in L^t$ as in the growth assumptions \eqref{F-growth2}. For this reason and for the sake of completeness of the discussion we provide the proof of Lemma~\ref{new-lem-2.4} in the Appendix.
\begin{rem}\label{Remark-K-and-norm}
 The quasiminimizing constant $K'$ in Lemma~\ref{new-lem-2.4} depends, among other parameters, on the Sobolev norm of $u$. Such a dependence is not a novelty and can be found in the literature, see e.g discussion at (2.1) in Giaquinta-Giusti \cite{Gia-Giu}, Theorem 6.1 in Giusti~\cite{Giusti} and the proof of Theorem 2.1 in Fan--Zhao~\cite{Fan-Z}. In a consequence some constants in Theorems~\ref{new-prop-4.2} and \ref{new-Theorem-4-1} may depend on $\|u\|_{\Wpxloc(\Om)}$ as well.
 However, this does not affect the validity of the local H\"older continuity result in Theorem~\ref{new-Theorem-4-1}.
 In general, one can eliminate the dependence of $K'$ on $\|u\|_{\Wpxloc}(\Om)$ by considering a family of uniformly bounded quasiminimizers (cf. Chiad\`o Piat--Coscia~\cite{Chiado-Coscia}).
\end{rem}

The next lemma will be needed to conclude the proof of Theorem~\ref{new-prop-4.2}, see also Remark 2.5 in \cite{Chiado-Coscia}. As in the case of Lemma~\ref{new-lem-2.4} we could not find a proof of this result in the literature and, therefore, decided to present the complete argument, see the Appendix.

\begin{lem}\label{new-ren-2.5}
 If $u$ is a bounded $K$-quasiminimizer of $\calFone_{, \Om}$, then so is $-u$. Moreover, let $\Om'\Subset \Om$. Then $u-k$ is a bounded $K_1$-quasiminimizer of the following energy functional $\calFtwo_{, \Om'}$ for $k\leq \sup_{\Om'}|u|$:
\begin{equation}\label{new-ren-enrg}
 \calFtwo_{, \Om'}(u):=\int_{\Om}(\mu|\nabla u|^{\px}+b(x)|u|^{\rx}+g(x)+\tilde h(x))\,dx,
\end{equation}
for $K_1=2^{r^+}(2^{r^+}+1)K$ and $\tilde h(x)=c(K, r^{+}, \sup_{\Om'}|u|) b(x)$ with $c=\frac{2^{r+}[(2^{r^+}+1)K+1](1+\sup_{\Om'}|u|)^{r^+}}{2^{r+}(2^{r+}+1)K-1}$.
\end{lem}

\section{H\"older continuity of quasiminimizers}\label{Section3}

In this section we show the main result of the paper, namely the H\"older continuity of quasiminimizers of the energy functional \eqref{energy-fun} under the growth conditions \eqref{F-growth} and \eqref{F-growth2}, see Theorem~\ref{new-Theorem-4-1}. Its proof relies on a number of auxiliary results which we present first. In our approach we follow the steps of the reasoning presented in Chiad\`o Piat--Coscia~\cite{Chiado-Coscia}. However, our work extends \cite{Chiado-Coscia} as we now study energy functionals under the more general growth conditions.

\begin{lem}[Caccioppoli-type inequalities]\label{Caccioppoli}
Let $\p$ be a bounded continuous variable exponent, and let $u$ be a local $K$-quasiminimizer of $\calF_\Om$. Then for each
$x_0 \in \Omega$ there exists $R_0>0$ such that for all
$0<r < R \leq R_0$ and for all $k\geq 0$ we have
\begin{align*}
 \int_{B_r} |\nabla (u-k)_+|^{p(x)}\,dx & \leq C\left( \int_{B_R}
\left(\frac{u-k}{R-r}\right)^{p(x)}\,dx + (1+k^{p^+_{B_R}})|A(k,R)|^{1-\frac{1}{t}}\right), \\
 \int_{B_r} |\nabla (u-k)_+|^{p^-}\,dx &\leq C\left( \int_{B_R}
\left(\frac{u-k}{R-r}\right)^{p^+_{B_R}}\,dx + (1+k^{p^+_{B_R}})|A(k,R)|^{1-\frac{1}{t}} + 2|A(k,R)|\right).
\end{align*}
Here $C$ depends only on $p^+$, $\|b\|_{L^{\sigmadot}}$, $\|g\|_{L^t}$ and the quasiminimizing constant $K$.
\end{lem}
\begin{proof}
%As in the proof for minimizers, see Lemma 1 in Toivanen~\cite{Toivanen-2}, we fix an arbitrary $k\ge 0$ and write $A(R)$ and $A(\rho)$ for $A(k,R)$ and $A(k,\rho)$ for brevity.
 Let $\eta \in C^\infty_0(B_R)$ be a test function such that $0 \leq \eta \leq 1$, $\eta = 1$ in $B_r$, and $|\nabla\eta| \leq \frac{2}{R-r}$. Let $w = (u-k)_+ = \max\{u-k,0\}$ and $v = u-\eta w$. Note that $v \leq u$, and $v$ differs from $u$ at most in $A(k, R)$. By the structural conditions~\eqref{F-growth} and the quasiminimizing
property,
\begin{align}
\int_{A(k, R)} |\nabla u|^{p(x)}\,dx &\leq K \int_{A(k, R)} |\nabla v|^{p(x)}\,dx + \int_{A(k, R)}
\left( b(x) \left(|u|^{p(x)}+K|v|^{p(x)}\right)+(1+K)g(x)\right)\,dx \nonumber \\
&\leq K\left( \int_{A(k, R)} |\nabla v|^{p(x)}\,dx + \int_{A(k, R)}
\left( b(x) \left(|u|^{p(x)}+|v|^{p(x)}\right)+2g(x)\right)\,dx \right). \label{Cacc-ineq1}
\end{align}
In $A(k, R)$ we have $u = u(1-\eta) + \eta(w+k)$, $v = u(1-\eta) + \eta k$, and hence $\nabla v =
(1-\eta)\nabla u - (u-k)\nabla\eta$. It follows that in $A(k, R)$ we have
\begin{equation}
|u|^{p(x)} + |v|^{p(x)} \leq C\left( (\eta w)^{p(x)} + |u|^{p(x)}(1-\eta)^{p(x)} +
\eta^{p(x)}k^{p(x)}\right)\label{Cacc-ineq2}
\end{equation}
and
\begin{equation}
|\nabla v|^{p(x)} \leq C\left( (1-\eta)^{p(x)} |\nabla u|^{p(x)} +
|\nabla\eta|^{p(x)} (u-k)^{p(x)} \right). \label{Cacc-ineq3}
\end{equation}
Thus by adding $\int_{A(k, R)} b(x)|u|^{p(x)}\,dx$ to the both sides of \eqref{Cacc-ineq1} and using \eqref{Cacc-ineq2}, \eqref{Cacc-ineq3} we obtain that
\begin{align*}
\begin{split}
\int_{A(k, R)} |\nabla u|^{p(x)} + b(x)|u|^{p(x)} &\leq C_0K\left( \int_{A(k, R)} (1-\eta)^{p(x)} \left( |\nabla u|^{p(x)} + b(x)|u|^{p(x)}\right)\right.\\
&+ \left.\int_{A(k, R)}\left(\frac{u-k}{R-r}\right)^{p(x)} + \int_{A(k, R)} \left(b(x)(\eta w)^{p(x)} + b(x)k^{p(x)}+ g(x)\right)\right)
\end{split}
\end{align*}
for some constant $C_0=C_0(p^+)$. We include $K$ in the constant $C_0$, and from this point on the proof proceeds as in Toivanen \cite[Lemma 1]{Toivanen-2} and \cite[Remark 2]{Toivanen-2}.

In order to prove the second inequality we note that for any $\xi \in \Rn$ it holds that
\[
|\xi|^{p^-} - 1 \leq |\xi|^{p(x)} \leq |\xi|^{p^+} + 1.
\]
Therefore,
\begin{align*}
 \int_{A(k, r)} \left(|\nabla u|^{p^-}-1\right) &\leq \int_{A(k, r)} |\nabla u|^{\px}\\
 &\leq C\left( \int_{B_R} \left(\frac{u-k}{R-r}\right)^{p^+_{B_R}}\,dx + (1+k^{p^+_{B_R}})|A(k,R)|^{1-\frac{1}{t}} + \int_{A(k,R)}1\right).
\end{align*}
The second Caccioppoli estimate follows immediately from this inequality.
\end{proof}

We recall the following Sobolev-Poincar\'e inequality with a variable exponent adapted to our setting and notation, see Proposition 3.1 in  \Chiado Piat--Coscia \cite{Chiado-Coscia}.
\begin{prop}\label{CPC-Prop-3.1}
  Let $p$ be a variable exponent satisfying assumptions \eqref{p-assump}. Then for every $M>0$ there exists a positive radius $R_1 = R_1(M,n,s,\|p\|_{L^s})$ such that for every $\gamma > 1/n - 1/s>0$ there exist two positive constants $\chi = \chi(n,p^-,s,\gamma,\|p\|_{L^s})$ and $c = c(n,p^-,p^+)$ for which the following inequality holds:
\[
\left( \vint_{B_R} \left| \frac{u}{R} \right|^{p(x)\frac{n}{n-1}}\,dx \right)^{\frac{n-1}{n}}
\leq c\,\vint_{B_R} |\nabla u|^{p(x)}\,dx + \chi|\{x\in B_R\,:\,|u|>0\}|^\gamma
\]
for every $B_R \subset \Omega$ with $0 < R \leq R_1$, and every $u \in W^{1,\p}_{0}(B_R)$ such that $\sup_{B_R} |u| \leq M$.
%\[
%\int_{B_R} |\nabla u|^{p(x)}\,dx < \infty\qquad \hbox{and} \qquad .
%\]
%and $u=0$ on $\partial B_R$.
\end{prop}

Our next result is the local supremum estimate for quasiminimizers of $\calF_{\Om}$. It improves and refines the estimate in \cite[Theorem 1]{Toivanen-2} by letting the right-hand side of \eqref{new-prop-4.2-est} to depend on a ratio of the measure of the level set and the measure of a ball, and also by introducing the dependence on $R^{p^-/p^+}$, cf. \cite[Formula (4.3)]{Chiado-Coscia}. Estimate \eqref{new-prop-4.2-est} will play a fundamental role in showing Theorem~\ref{new-Theorem-4-1}. In the proof below we use Lemmas~\ref{new-lem-2.4} and \ref{new-ren-2.5}, see Preliminaries and Appendix.

In the theorem below we will slightly abuse the notation and for the sake of its simplicity denote $p^-:=p^{-}_{B_R}$ and $p^+:=p^{+}_{B_R}$.

%[$\approx$ Prop 4.2; fine $L^\infty$ estimate; FORMULATE EXACT STATEMENT LATER]
\begin{thm}\label{new-prop-4.2}
 Let variable exponent $\p$ satisfy \eqref{p-assump} and let $u\in \Wpxloc(\Om)$ be a $K$-quasiminimizer of energy \eqref{energy-fun} under the growth conditions \eqref{F-growth} and \eqref{F-growth2} with the additional assumption that $t>n$. If $|k|\leq\sup |u|$, then for every $B_R \subset \Omega$ with $R < R_1$ we have
\begin{equation}\label{new-prop-4.2-est}
\sup_{B_{R/2}} (u-k) \leq c R^{\frac{p^-}{p^+}} \left(  \left( \frac{|A(k,R)|}{R^n} \right)^\beta \vint_{A(k,R)} \left| \frac{u-k}{R} \right|^{p(x)}\,dx + \frac{1}{R^\frac{n}{t}} \right)^{\frac{1}{p^+}},
\end{equation}
where $\beta>0$ and satisfies $\beta(1+\beta - 1/t) = 1/n$, while $c = c(n,p^-,p^+,s)$.
\end{thm}

\begin{rem}
 In Theorem~\ref{new-prop-4.2} we assume that $g\in L^t(\Om)$ for $t>n$. Such an assumption is needed in order to obtain estimate~\eqref{ineq:4.2-v}, crucial for the de Giorgi iteration.
\end{rem}

In the proof of Theorem~\ref{new-prop-4.2} we will apply the iteration scheme of the de Giorgi method. The following lemma will be necessary for the application of this technique, cf. \cite[Lemma~7.1]{Giusti}.
\begin{lem}\label{lem:deGiorgi}
 Let $\alpha>0$ and $\{x_i\}$ be a sequence of real positive numbers, such that
 $x_{i+1}\leq C B^i x_i^{1+\alpha}$ with $B>1$ and $C>0$. If $x_0\leq C^{-\frac{1}{\alpha}}B^{-\frac{1}{\alpha^2}}$, then $\lim_{i\to \infty}x_i=0$.
\end{lem}

\begin{proof}[Proof of Theorem~\ref{new-prop-4.2}]
Fix $h<k$, $R/2 \leq \rho < \sigma \leq R < \min\{1,R_1\}$, where $R_1$ is as in Proposition \ref{CPC-Prop-3.1}, and let $\psi \in C_0^\infty(B_{(\sigma+\rho)/2})$ be a cut-off function with $\psi = 1$ in $B_\rho$ and $|\nabla \psi| \leq \frac{4}{\sigma-\rho}$. Since $\frac{1}{n} - \frac{1}{s} < 1$, we apply the H\"older inequality and the Sobolev-Poincar\'e inequality of Proposition \ref{CPC-Prop-3.1} with $\gamma=1$, and obtain
\begin{align*}
\vint_{A(k,\rho)} \left|\frac{u-k}{\rho}\right|^{p(x)}\,dx &\leq c(n,p^+) \left(\frac{|A(k,\rho)|}{R^n}\right)^{1/n} \left( \vint_{B_R} \left|\frac{(u-k)^+\psi}{R}\right|^{p(x)\frac{n}{n-1}}\,dx \right)^{\frac{n-1}{n}}  \\
& \leq c \left(\frac{|A(k,\rho)|}{R^n}\right)^{1/n} \left(\vint_{B_R} |\nabla ((u-k)^+\psi)|^{p(x)}\,dx + |B_{\frac{\sigma+\rho}{2}}\cap  \{|(u-k)^+\psi| > 0 \} | \right),
\end{align*}
where $c = c(n,p^-,p^+,s,\|p\|_{1,s})$. Observe that in the set $A(k,\frac{\sigma+\rho}{2})$ it holds that
\[
|\nabla ((u-k)^+\psi)| \leq |\nabla u| + 4 \left|\frac{u-k}{\sigma-\rho}\right|.
\]
Therefore, we have the following inequality
\begin{align}
\vint_{A(k,\rho)} \left|\frac{u-k}{\rho}\right|^{p(x)}\!\! & \leq c \left(\frac{|A(k,\rho)|}{R^n}\right)^{1/n} \left(\vint_{A(k,\frac{\sigma+\rho}{2})} \bigg(|\nabla u|^{p(x)} + \left|\frac{u-k}{\sigma-\rho}\right|^{p(x)} \bigg)\!dx + |A(k,\frac{\sigma+\rho}{2})| \right). \label{4.2-i}
\end{align}
Next, we apply the Caccioppoli inequality in Lemma \ref{Caccioppoli} by choosing $r := \frac{\sigma+\rho}{2}$ and $R := \sigma$
\begin{equation}
\vint_{A(k,\frac{\sigma+\rho}{2})} \left|\nabla u\right|^{p(x)} \leq c(p^+,\|b\|_t,\|g\|_t) \left(\vint_{A(k,\sigma)} \left|\frac{u-k}{\sigma-\rho}\right|^{p(x)}\,dx + (1+k^{p^+_{B_\sigma}})|A(k,\sigma)|^{1-\frac{1}{t}} \right).\label{4.2-ii}
\end{equation}
Combining \eqref{4.2-i} and \eqref{4.2-ii}, we obtain the following inequality
\begin{align}
\vint_{A(k,\rho)} \left|\frac{u-k}{\rho}\right|^{p(x)} \leq c \left(\frac{|A(k,\sigma)|}{R^n}\right)^{\frac{1}{n}} \Bigg(&\frac{1}{|B_{\frac{\sigma+\rho}{2}}|} \int_{A(k,\sigma)} \left|\frac{u-k}{\sigma-\rho}\right|^{p(x)} \,dx \nonumber \\
&+ |A(k,\frac{\sigma+\rho}{2})| + \frac{1+k^{p^+_{B_\sigma}}}{|B_{\frac{\sigma+\rho}{2}}|} |A(k,\sigma)|^{1-\frac{1}{t}} \Bigg), \label{4.2-iii}
\end{align}
where $c$ depends on the parameters of the preceding constants. Since $R/2 \leq \rho < \sigma \leq R$, we have that
$|B_\sigma|/|B_{\frac{\sigma+\rho}{2}}| \leq 2^n$. Moreover, as $\sigma \leq R < 1$, it holds that
$|A(k,\sigma)|/\omega_n \leq 1$, where $\omega_n$ stands for the measure of the unit $n$-dimensional ball. Hence,
by including $2^n$ into constant $c$, we have that estimate \eqref{4.2-iii} takes the form
\begin{align}
\vint_{A(k,\rho)} \left|\frac{u-k}{\rho}\right|^{p(x)}\!\! \leq c \left(\frac{|A(k,\rho)|}{R^n}\right)^{\frac{1}{n}} \Bigg[
&\vint_{A(k,\sigma)} \left|\frac{u-k}{\sigma-\rho}\right|^{p(x)} \,dx \nonumber \\
&+\frac{1}{|B_\sigma|} \Big( |B_\sigma| + \frac{|B_\sigma|}{|B_{\frac{\sigma+\rho}{2}}|} \left(1+k^{p^+_{B_\sigma}}\right) \Big)\left( |A(k,\sigma)| + |A(k,\sigma|^{1-\frac{1}{t}} \right)
\Bigg] \nonumber \\
\leq c \left(\frac{|A(k,\rho)|}{R^n}\right)^{\frac{1}{n}} & \!\left(
\vint_{A(k,\sigma)} \left|\frac{u-k}{\sigma-\rho}\right|^{p(x)}\! dx
+ \frac{2}{|B_\sigma|} \left( 1 + 2^n \left(1+k^{p^+_{B_\sigma}}\right) \right)
|A(k,\sigma|^{1-\frac{1}{t}}
\right). \label{ineq: 4.2}
\end{align}
We modify the constant on the right-hand side of \eqref{ineq: 4.2} by observing that $R^{-n}|A(k,\rho)| \leq \sigma^{-n} |A(k,\sigma)|$. Furthermore, by taking into account that
\[
\vint_{A(k,\sigma)} \left|\frac{u-k}{\sigma-\rho}\right|^{p(x)}\,dx \leq \vint_{A(h,\sigma)} \left|\frac{u-h}{\sigma-\rho}\right|^{p(x)}\,dx
\]
for $h<k$, we get for $\beta>0$
\begin{align}
\left(\frac{|A(k,\rho)|}{R^n}\right)^{\beta} \vint_{A(k,\rho)} \left|\frac{u-k}{\rho}\right|^{p(x)}
\leq c \left(\frac{|A(k,\sigma)|}{R^n}\right)^{\frac{1}{n}} \Bigg[
&\left(\frac{|A(k,\rho)|}{\sigma^n}\right)^{\beta} \vint_{A(h,\sigma)}\left(\frac{\sigma}{\sigma-\rho}\right)^{\px} \left|\frac{u-h}{\sigma}\right|^{p(x)}dx \nonumber \\
&+ \frac{2+k^{p^+_{B_\sigma}}}{\sigma^{n\beta+n}} |A(k,\sigma)|^{1+\beta-\frac{1}{t}}\Bigg]. \label{previous}
\end{align}
Observe that for $\sigma<1$ and $0<k-h<1$ we have
\begin{align}
|A(k,\sigma)| &= \int_{B_\sigma \cap \{u > k\}} 1\,dx \leq \int_{A(h,\sigma)} \left|\frac{u-h}{k-h}\right|^{p(x)}\,dx
= \int_{A(h,\sigma)} \left(\frac{\sigma}{k-h}\right)^{p(x)} \left|\frac{u-h}{\sigma}\right|^{p(x)}\,dx \nonumber \\
&\leq \frac{\sigma^{p^-}}{(k-h)^{p^+}} \int_{A(h,\sigma)}  \left|\frac{u-h}{\sigma}\right|^{p(x)}\,dx. \label{ineq2:4.2}
\end{align}
We are now in a position to show the key-point estimate of this proof. First, we combine observation in \eqref{ineq2:4.2} with \eqref{previous} to arrive at the following estimate:
\begin{align*}
\left(\frac{|A(k,\rho)|}{R^n}\right)^{\beta}\!\! \vint_{A(k,\rho)} \left|\frac{u-k}{\rho}\right|^{p(x)}\! & \leq \! c \left(\frac{|A(k,\sigma)|}{R^n}\right)^{\frac{1}{n}}
\Bigg[\left(\frac{|A(k,\sigma)|}{\sigma^n}\right)^{\beta}\!\left( \frac{\sigma}{\sigma-\rho}\right)^{p^+}
\!\vint_{A(h,\sigma)} \left|\frac{u-h}{\sigma}\right|^{p(x)} \nonumber \\
&+ \frac{2+k^{p^+_{B_\sigma}}}{\sigma^{n\beta+n}} \left( \frac{\sigma^{p^-}}{(k-h)^{p^+}}\right)^{1+\beta-\frac{1}{t}}
\!\!\sigma^{n + n\beta - \frac{n}{t}} \left( \vint_{A(h,\sigma)} \left|\frac{u-h}{\sigma}\right|^{p(x)} \right)^{1+\beta-\frac{1}{t}} \Bigg] \nonumber \\
&:= I_1+I_2.
\end{align*}
We have introduced integrals $I_1$ and $I_2$ in order to simplify the presentation. Next, we refine further the estimate for $I_1$.
\begin{align}
I_1 &\leq c \left(\frac{|A(k,\rho)|}{R^n}\right)^{\frac{1}{n}}
\left( \frac{\sigma}{\sigma-\rho}\right)^{p^+} \left(\frac{|A(h,\sigma)|}{\sigma^n}\right)^{\frac{1}{t}}
\frac{|A(k,\sigma)|^{\beta-\frac{1}{t}}}{\sigma^{n\beta-\frac{n}{t}}}
\vint_{A(h,\sigma)} \left|\frac{u-h}{\sigma}\right|^{p(x)} \,dx \nonumber \\
&\leq c \left(\frac{|A(k,\rho)|}{R^n}\right)^{\frac{1}{n}}
\left( \frac{\sigma}{\sigma-\rho}\right)^{p^+} \frac{1}{\sigma^{n\beta-\frac{n}{t}}}
\left(\frac{\sigma^{p^-}}{(k-h)^{p^+}}\right)^{\beta-\frac{1}{t}} \left(\vint_{A(h,\sigma)} \left|\frac{u-h}{\sigma}\right|^{p(x)} \,dx\right)^{1+\beta-\frac{1}{t}}\sigma^{n(\beta-\frac{1}{t})}, \nonumber
\end{align}
where we also use the observation that $\left(\frac{|A(h,\sigma)|}{\sigma^n}\right)^{\frac{1}{t}}\leq 1$. We join together the estimate for $I_1$ and $I_2$ and obtain that
\begin{align}
& \left(\frac{|A(k,\rho)|}{R^n}\right)^{\beta}\!\! \vint_{A(k,\rho)} \left|\frac{u-k}{\rho}\right|^{p(x)} \nonumber \\
&\leq C \left(\frac{|A(h,\sigma)|}{\sigma^n}\right)^{\frac{1}{n}}
\left[
\left(\frac{|A(k,\sigma)|}{\sigma^n}\right)^{\beta}
\left(\vint_{A(h,\sigma)} \left|\frac{u-h}{\sigma}\right|^{p(x)} \,dx \right)\right]^{1+\beta-\frac{1}{t}}
\left(\frac{|A(h,\sigma)|}{\sigma^n}\right)^{-\beta(1+\beta-\frac{1}{t})}, \label{4.2-iv}
\end{align}
where expression $C$ is independent of $u$ and is defined as follows:
\begin{equation}
C:=c \left\{\left( \frac{\sigma}{\sigma-\rho}\right)^{p^+}
\left(\frac{\sigma^{p^-}}{(k-h)^{p^+}}\right)^{\beta-\frac{1}{t}}
+ \frac{2+k^{p^+_{B_\sigma}}}{\sigma^{\frac{n}{t}}}
\left( \frac{\sigma^{p^-}}{(k-h)^{p^+}}\right)^{1+\beta-\frac{1}{t}}
\right\}.\label{const-4.2}
\end{equation}
We choose $\beta > 0$ such that $\frac{1}{n} - \beta(1+\beta-\frac{1}{t}) = 0$. Upon solving this equation we have that
\[
0 < \beta^+ = \frac{-(1-\frac{1}{t}) + \sqrt{(1-\frac{1}{t})^2 + \frac{4}{n}}}{2} \leq \frac{2}{n} \frac{t}{t-1}.
\]
Furthermore, since by assumptions $t>n$, it holds that $\beta^+ > \frac{1}{t}$. Therefore, estimate \eqref{4.2-iv}
takes the form:
\begin{equation}
 \left(\frac{|A(k,\rho)|}{R^n}\right)^{\beta}\!\! \vint_{A(k,\rho)} \left|\frac{u-k}{\rho}\right|^{p(x)}
\leq C \left[
\left(\frac{|A(k,\sigma)|}{\sigma^n}\right)^{\beta}
\left(\vint_{A(h,\sigma)} \left|\frac{u-h}{\sigma}\right|^{p(x)} \,dx \right)\right]^{1+\beta-\frac{1}{t}}. \label{ineq:4.2-v}
\end{equation}
Our next goal is to apply the iteration scheme in the de Giorgi method, see Lemma~\ref{lem:deGiorgi} and cf. Lemma~7.1 in Giusti~\cite{Giusti}. In order to do so, we define the following families of radii and level sets:
\begin{align*}
\sigma &= R_i := \frac{R}{2} + \frac{R}{2^{i+1}}, \qquad \rho = R_{i+1}, \\
h &= k_i := d R^{\frac{p^-}{p^+}} \left(1 - \frac{1}{2^i}\right), \qquad k = k_{i+1}.
\end{align*}
for every $i \in \N$, and with some $d>0$ to be determined later. Note that
\begin{align*}
k_{i+1} - k_i = \frac{d R^{\frac{p^-}{p^+}}}{2^{i+1}},\qquad R_i - R_{i+1} = \frac{R}{2^{i+2}}.
\end{align*}
Note that both differences are less than $1$ for a fixed $R$ and large enough $i$. This justifies our assumptions at \eqref{ineq2:4.2}. With this notation we complete the preparation for the iteration scheme at \eqref{ineq:4.2-v}.
\begin{align*}
\frac{\sigma}{\sigma-\rho} &= \frac{\frac{R}{2}(1 + \frac{1}{2^{i+1}})}{\frac{R}{2^{i+1}}} \leq 2^i + 1 \leq 2^{i+1}, \\
\frac{\sigma^{p^-}}{(k-h)^{p^+}} &= \left(\frac{R}{2}\right)^{p^-}\frac{\left(1+\frac{1}{2^{i+1}}\right)^{p^-}}{\left(\frac{dR^{\frac{p^-}{p^+}}}{2^{i+1}}\right)^{p^+}} = \left(\frac{2^{i+1} + 1}{2^{i+1}}\right)^{p^-} \frac{2^{(i+1)p^+}}{2^{p^-}} \frac{1}{d^{p^+}} \leq \frac{2^{(i+1)p^+}}{d^{p^+}}.
\end{align*}
This, together with the assumption that $k\leq \sup_{\Om}|u|$ imply that constant $C$ in \eqref{const-4.2} can be estimated as follows
\begin{align*}
C \leq & 2^{(i+1)p^+} \left(\frac{2^{(i+1)p^+}}{d^{p^+}}\right)^{\beta-\frac{1}{t}} + 2^{\frac{n}{t}}\frac{2 + (\sup_\Omega |u|)^{p^+}}{R^{\frac{n}{t}}} \left(\frac{2^{(i+1)p^+}}{d^{p^+}}\right)^{1+\beta-\frac{1}{t}} \\
& \leq 2^{\frac{n}{t}} \frac{2^{(i+1)p^+(1+\beta-\frac{1}{t})}}{d^{p^+(\beta-\frac{1}{t})}}
\left(1 + \frac{1}{d^{p^+}R^{n/t}} \right)\left(3 + \sup_\Omega |u| \right)^{p^+}.
\end{align*}
We choose $d$ in such a way that $d^{p^+}R^{n/t} \geq 1$. Finally, let us define
\[
\phi(k,\rho) := \left(\frac{|A(k,\rho)|}{R^n}\right)^\beta \vint_{A(k,\rho)} \left| \frac{u-k}{\rho} \right|^{a(x)}\,dx.
\]
With this notation \eqref{ineq:4.2-v} reads:
\begin{align*}
\phi(k_{i+1},R_{i+1}) \leq c \frac{2^{(i+1)p^+(1+\beta-1/t)}}{d^{p^+(\beta-1/t)}}
\underbrace{\left( 1 + \frac{1}{d^{p^+}R^{n/t}} \right)}_{\leq 2}
\phi(k_i,R_i)^{1+\beta-1/t}.
\end{align*}
Here, constant $c$ depends additionally on $\sup_\Omega |u|$.

We are in a position to apply the iteration lemma (Lemma~\ref{lem:deGiorgi}) with
\begin{align*}
C= c \frac{2^{p^+(1+\beta-\frac{1}{t})}}{d^{p^+(\beta-\frac{1}{t})}} > 0 \qquad
B &= 2^{p^+(1+\beta-\frac{1}{t})} > 1 \quad \textrm{ and} \quad \alpha = \beta - \frac{1}{t}.
\end{align*}
The condition $\phi_0 = \phi(k_0,R_0) = \phi(0,R) \leq C^{-1/\alpha} B^{-1/\alpha^2}$ reads in our case as follows:
\[
d^{p^+} \geq c 2^{p^+(1+\frac{1}{\beta-1/t})^2} \phi(0,R).
\]
By applying Lemma~\ref{lem:deGiorgi}, we obtain that
\[
\lim_{i\to\infty} \phi(k_i,R_i) = \phi(\alpha R^{p^-/p^+},R/2) = 0.
\]
We take $d$ to be defined by $d^{p^+}= \frac{1}{R^{n/t}} + c 2^{p^+(1+\frac{1}{\beta-1/t})^2} \phi(0,R)$.
Thus
\[
\sup_{B_{R/2}} u \leq d R^{p^-/p^+} = c R^{p^-/p^+} \left( \left(\frac{|A(0,R)|}{R^n}\right)^{\beta} \vint_{A(0,R)} \left|\frac{u}{R}\right|^{p(x)}\,dx + \frac{1}{R^{n/t}} \right)^{1/p^+}.
\]
In order to show the assertion of theorem for $u-k$ we appeal to Lemmas~\ref{new-lem-2.4} and \ref{new-ren-2.5} and obtain that $u-k$ is a quasiminimizer of an auxiliary energy $\calFtwo$, see \eqref{new-ren-enrg}. We use this observation to obtain the Caccioppoli-type estimate as in Lemma~\ref{Caccioppoli} with constant $C$ modified according to coefficients of $\calFtwo$, cf. estimate \eqref{Cacc-ineq1}. Then, we repeat the above reasoning and, hence, starting from inequality \eqref{4.2-ii}, constants in estimates in the above proof begin to depend additionally on $\|u\|_{\Wpxloc(\Om)}$ and functions $h, \tilde h$. (Note that the latter two functions are expressed in terms of function $b$.) The final sup-estimate is obtained following the same lines as in the case of $k=0$ completing the proof of Theorem~\ref{new-prop-4.2}.
\end{proof}

Next lemma provides an estimate for the amount of the level set contained in a given ball in terms of the oscillations. The lemma is a generalization of the similar technical result from \Chiado Piat-Coscia \cite{Chiado-Coscia}, cf. Lemma 4.4. The fact that we now allow more general coefficients $b$ and $g$ than in \cite{Chiado-Coscia} results in an additional oscillation term in the assertion \eqref{osc-assertion}.

\begin{lem}\label{new-lemma-4.4}

 Let $\p$ satisfy assumptions \eqref{p-assump} and $u$ be a local quasiminimizer of \eqref{energy-fun}. Suppose that for a given ball $B_R$ such that $B_{2R} \subset \Omega$ and for $k_0 = \frac12(M(u,2R) + m(u,2R))$ there exists a constant $\delta < 1$ for which $|A(k_0, R)| \leq \delta|B_R|$. Moreover, let us assume that there exist $\nu \in \N$ and $\eta \geq 1$ such that $\osc(u,2R) \geq 2^{\nu+1} \eta R^{p^-/p^+}$. Then
\begin{equation}\label{osc-assertion}
\frac{|A(k_\nu,R)|}{R^n} \leq c \frac{R^{\frac{p^--p^+}{p^-}}}{\nu}
\Big( \osc(u,2R)^{\frac{p^+}{p^-}-1}+ \osc(u,2R)^{\left(\frac{p^+}{p^-}\right)^2-1}
+ \tilde{c}\osc(u,2R)^{\frac{p^+}{(p^-)^2}\left(p^+ - \frac{n}{t}\right)}
\Big).
\end{equation}
Here $k_\nu = M(u,2R) - \frac{\osc(u,2R)}{2^{\nu+1}}$ and $c$ depends on $n,p^-,p^+,\eta, \sup_\Omega |u|$ as well as on $\|b\|_{L^{\sigmadot}}, \|g\|_{L^t}$, whereas
$\tilde{c}=2^{(\nu+1)\left( 1 - \frac{p^+}{(p^-)^2} (p^+ - \frac{n}{t})\right)}$.
\end{lem}

\begin{rem}\label{new-rem-4.5}

Notice that a bounded Lipschitz variable exponent satisfies the log-H\"older continuity condition. Thus, by assumptions \eqref{p-assump} on $\p$, the following well-known property holds (see Lemma 4.1.6 in \cite{DHHR})
\[
\lim_{R\to 0^{+}} R^{p^-_{B_{2R}}-p^+_{B_{2R}}}\,=\,0.
\]
In consequence, Lemma~\ref{new-lemma-4.4} implies that
\[
\frac{|A(k_\nu,R)|}{R^n} \leq \frac{\tilde{C}}{\nu},
\]
for $R$ small enough, where $\tilde{C}$ depends among other parameters on $\sup_\Omega |u|$.
\end{rem}

\begin{proof}[Proof of Lemma~\ref{new-lemma-4.4}]
We follow the steps of the proof for Lemma 4.4 in \Chiado Piat--Coscia~\cite{Chiado-Coscia}. Given $h,k$ with $k_0 < h < k \leq M(u,2R)$ we define
\begin{align*}
v(x) = \begin{cases} k-h & \textrm{ if } u\geq k \\ u-h & \textrm{ if } h < u < k \\ 0 & \textrm{ if } u\leq h. \end{cases}
\end{align*}
 Note that $v \in W^{1,1}_\textrm{loc}(\Omega)$ and $v=0$ in $B_R\setminus A(k_0,R)$. Moreover, $|B_R\setminus A(k_0,R)|>(1-\delta)|B_R|$. Thus, we can apply the classical Sobolev-Poincar\'e inequality to get
\[
\left( \int_{B_R} v^{\frac{n}{n-1}}\,dx \right)^{\frac{n-1}{n}} \leq c(n) \int_\Delta |\nabla v|\,dx = c(n) \int_\Delta |\nabla u|\,dx,
\]
where $\Delta = A(h,R) \setminus A(k,R)$. Next, by the constant exponent H\"older inequality, we have
\[
(k-h) |A(k,R)|^{\frac{n-1}{n}} \leq \left( \int_{B_R} v^{\frac{n}{n-1}}\,dx\right)^{\frac{n-1}{n}} \leq c|\Delta|^{\frac{p-1}{p}} \left( \int_{A(h,R)} |\nabla u|^p\,dx \right)^{1/p}.
\]
  The second Caccioppoli-type inequality in Lemma \ref{Caccioppoli} together with observation $|A(h,2R)| \leq 2^n R^n$ imply that
\begin{align}
&(k-h) |A(k,R)|^{\frac{n-1}{n}} \nonumber \\
 & \leq c |\Delta|^{\frac{p^--1}{p^-}} \left( \int_{A(h,2R)} \left|\frac{u-h}{R}\right|^{p^+} + (1 + h^{p^+_{B_{2R}}})|A(h,2R)|^{1-\frac{1}{t}} + |A(h,2R)| \right)^{\frac{1}{p^-}} \nonumber \\
 & \leq c |\Delta|^{\frac{p^--1}{p^-}} \left( |A(h,2R)| \frac{|M(u,2R)-h|^{p^+}}{R^{p^+}}
+ (1+ h^{p^+_{B_{2R}}})|A(h,2R)|^{1-\frac{1}{t}} + |A(h,2R)| \right)^{\frac{1}{p^-}} \nonumber \\
& \leq c |\Delta|^{\frac{p^--1}{p^-}} R^{\frac{n}{p^-}(1-\frac{1}{t}) - \frac{p^+}{p^-}}
\left( |A(h,2R)|^{\frac{1}{t}} |M(u,2R)-h|^{p^+} + (1 + h^{p^+_{B_{2R}}}) R^{p^+} + R^{p^+} |A(h,2R)|^{\frac{1}{t}} \right)^{\frac{1}{p^-}} \nonumber \\
& \leq c |\Delta|^{\frac{p^--1}{p^-}} R^{\frac{1}{p^-} \left(n(1-\frac{1}{t})- p^+ + \frac{n}{t}\right)}
\left( |M(u,2R)-h|^{p^+} + (1 + h^{p^+_{B_{2R}}}) R^{p^+ - \frac{n}{t}} + R^{p^+} \right)^{\frac{1}{p^-}}. \label{4.4-i}
\end{align}
We take $k =k_i:= M(u,2R) - \frac{\osc(u,2R)}{2^{i+1}}$, $h = k_{i-1}$ and $\Delta_i := A(k_{i-1},R)\setminus A(k_i,R)$. Note that with these choices $k-h=\frac{\osc(u,2R)}{2^{i+1}}$. Furthermore, under the assumption that $\osc(u,2R) \geq 2^{\nu+1} \eta R^{p^-/p^+}$ for some $\nu \in \N$ and $\eta \geq 1$, we have
\begin{equation*}
R \leq \left(\frac{\osc(u,2R)}{2^{\nu+1}\eta}\right)^{\frac{p^+}{p^-}}.
\end{equation*}
Thus \eqref{4.4-i} becomes
\begin{align*}
&\frac{\osc(u,2R)}{2^{i+1}} |A(k_i,R)|^{\frac{n-1}{n}} \\
& \leq C |\Delta_i|^{\frac{p^--1}{p^-}} R^{\frac{n-p^+}{p^-}}  \Bigg[\left(\frac{\osc(u,2R)}{2^i}\right)^{p^+} + \Big(1 +\sup_\Omega |u|^{p^+_{B_{2R}}} \Big) \left(\frac{\osc(u,2R)}{2^{\nu+1}\eta}\right)^{\frac{p^+}{p^-}\left(p^+ - \frac{n}{t}\right)} \\
&\phantom{AAAAAAAAAAAAa}+ \left(\frac{\osc(u,2R)}{2^{\nu+1}\eta}\right)^{\frac{(p^+)^2}{p^-}}
\Bigg]^{\frac{1}{p^-}}.
\end{align*}
Note that for every $i \in \N$ such that $i \leq \nu$, we have $|A(k_\nu,R)|^{\frac{n-1}{n}} \leq |A(k_i,R)|^{\frac{n-1}{n}}$. Moreover, $|A(k_\nu,R)|^{1/n} \leq cR$ and $|\Delta_i|^{\frac{p^--1}{p^-}} \leq c R^{n - \frac{n}{p^-}}$. Thus
\begin{align*}
|A(k_\nu,R)| \leq 2 c R^{1 + n - \frac{n}{p^-} + \frac{n-p^+}{p^-}}
& \Bigg[\left(\frac{\osc(u,2R)}{2^i}\right)^{\frac{p^+}{p^-}-1} + \Big(1 +\sup_\Omega |u|^{\frac{p^+_{B_{2R}}}{p^-}} \Big) \left(\frac{\osc(u,2R)}{2^{\nu+1}\eta}\right)^{\frac{p^+}{(p^-)^2}\left(p^+ - \frac{n}{t}\right)-1} \\
&+ \left(\frac{\osc(u,2R)}{2^{\nu+1}\eta}\right)^{\frac{(p^+)^2}{(p^-)^2}-1}\Bigg].
\end{align*}
We take the sum over $i=1, \ldots, \nu$, divide by $R^n$ and estimate $\sum_{i=1}^\nu 2^{(1-p^+/p^-)i} \leq c(p^+,p^-)$. As a result we obtain:
\begin{align*}
& \sum_{i=1}^\nu |A(k_\nu,R)| = \nu \frac{|A(k_\nu,R)|}{R^n} \\
& \leq c R^{\frac{p^--p^+}{p^-}} \Bigg\{\osc(u,2R)^{\frac{p^+}{p^-}-1} + \Big(1 + \sup_\Omega |u|^{\frac{p^+_{B_{2R}}}{p^-}}\Big)
\frac{\sum_{i=1}^\nu 2^i}{2^{(\nu+1)\frac{p^+}{(p^-)^2}(p^+ - \frac{n}{t})}} \osc(u,2R)^{\frac{p^+}{(p^-)^2}\left(p^+ - \frac{n}{t}\right)-1} \\
& +\frac{\sum_{i=1}^\nu 2^i}{2^{(\nu+1)\frac{(p^+)^2}{(p^-)^2}}} \osc(u,2R)^{\frac{(p^+)^2}{(p^-)^2}-1}
\Bigg\}.
\end{align*}
We conclude that
\begin{align*}
\frac{|A(k_\nu,R)|}{R^n} \leq c \frac{R^{\frac{p^--p^+}{p^-}}}{\nu}
\Big( \osc(u,2R)^{\frac{p^+}{p^-}-1}+ \osc(u,2R)^{\left(\frac{p^+}{p^-}\right)^2-1}
+ \tilde{c}\osc(u,2R)^{\frac{p^+}{(p^-)^2}\left(p^+ - \frac{n}{t}\right)}
\Big).
\end{align*}
Here $\tilde{c}:=2^{(\nu+1)\left( 1 - \frac{p^+}{(p^-)^2} (p^+ - \frac{n}{t})\right)}$ while $c$ also depends on $\sup_\Omega |u|$.
\end{proof}
\begin{thm}[H\"older continuity of quasiminimizers]\label{new-Theorem-4-1}
Let $p$ be a variable exponent satisfying \eqref{p-assump} and $u$ be a quasiminimizer of $\mathcal{F}_{\Om}$ defined at \eqref{energy-fun}--\eqref{F-growth2} under the additional assumption that $t > n$. Then $u$ is locally H\"older continuous in $\Omega$ with the exponent $0<\alpha<1$ depending on $p^-, p^+, t, n$.
\end{thm}
\begin{proof}
We use again the fact that $u$ is a quasiminimizer of $\mathcal{F}_1$, see Lemma \ref{new-lem-2.4}. Fix a ball $B_{2R} \subset \Omega$ with $2R < R_1$, let $k_0 = \frac12 ( M(u,2R) + m(u,2R))$. We can also assume $|A(k_0,R)| \leq \frac12 |B_R|$, as otherwise $|B_R| - |A(k_0,R)| \leq \frac12 |B_R|$, and all the following arguments would apply to $-u$.

Let us fix $k_i = M(u,2R) - \frac{\osc(u,2R)}{2^{i+1}}$, $i \in \N$. Then $k_i \geq k_0$ and $k_i$ increases to $M(u,2R)$ as $i\to\infty$. In particular, if we take $M = 2\sup_\Omega |u|$, Theorem \ref{new-prop-4.2} holds with $k$ replaced by $k_i$ for every $i \in \N$, giving
\[
\sup_{B_{R/2}} (u-k_i) \leq c R^{p^-/p^+} \left(  \left( \frac{|A(k,R)|}{R^n} \right)^\beta \vint_{A(k_i,R)} \left| \frac{u-k_i}{R} \right|^{p(x)}\,dx + R^n  \right)^{1/p^+}.
\]
Moreover, assuming for simplicity that $R<1$ and taking into account that
\[
\sup_{B_{R/2}} (u-k_i) \leq \frac{\osc(u,2R)}{2^{i+1}} \leq \frac{M}{2^{i+1}} < 1
\]
holds for $i \geq i_0$ for every $R$, we have the following estimate for every $i \geq i_0$ and every $0 < R < 1$:
\[
\sup_{B_{R/2}} (u-k_i) \leq c R^{\frac{p^-}{p^+}-\frac{n}{tp^+}} + c R^{\frac{p^-}{p^+}-1} \left(\frac{\osc(u,2R)}{2^{i+1}}\right)^{\frac{p^-}{p^+}} \left(\frac{|A(k_i,R)|}{R^n}\right)^{\frac{1+\beta}{p^+}}.
\]
We shall consider two cases. First, assume that for some $\nu \in \N$ to be chosen later
\begin{equation}\label{first-case}
\frac{\osc(u,2R)}{2^{\nu+1}} \geq R^{\frac{p^-}{p^+}}
\end{equation}
holds.
Then by Lemma \ref{new-lemma-4.4}, with $\eta = 1$, and by Remark \ref{new-rem-4.5},
\[
\sup_{B_{R/2}} (u-k_i) \leq c R^{\frac{1}{p^+}\left(1-\frac{n}{t}\right)} + c R^{\frac{p^-}{p^+}-1} \left(\frac{\tilde{C}}{\nu}\right)^{\frac{1+\beta}{p^+}} \left(\frac{\osc(u,2R)}{2^{i+1}}\right)^{\frac{p^-}{p^+}}.
\]
By the definition of $k_\nu$, this means that
\[
M(u,R/2) - M(u,2R) + \frac{\osc(u,2R)}{2^{\nu+1}} \leq c\left(\frac{\tilde{C}}{\nu}\right)^{\frac{1+\beta}{p^+}} R^{\frac{p^-}{p^+}-1} \left( \frac{\osc(u,2R)}{2^{\nu+1}} \right)^{\frac{p^-}{p^+}} + c R^{\frac{n+p^-}{p^+}}.
\]
If we subtract $m(u,R/2)$ and $m(u,2R)$ from both sides, then we get that
\begin{align*}
&\osc(u,R/2) \leq \left( 1 - \frac{1}{2^{i+1}} \right) \osc(u,2R) + c \left(\frac{\tilde{C}}{\nu}\right)^{\frac{1+\beta}{p^+}} R^{\frac{p^-}{p^+}-1} \left(\frac{\osc(u,2R)}{2^{i+1}}\right)^{\frac{p^-}{p^+}} + c R^{\frac{1}{p^+}\left(1-\frac{n}{t}\right)}.
\end{align*}
%Then we proceed exactly as in CPC, pg. 297, to obtain that \komT{Details must be added, especially that it is here that we used the log-H\"older continuity assumption.}
%%%%%%%%%%%%%%%%%%%%%%%%%
Note that \eqref{first-case} can be written as
\[
\frac{\osc(u,2R)}{2^{\nu+1}R^{\frac{p^-}{p^+}}} \geq 1.
\]
Thus, we have the estimate
\begin{align}\label{osc-and-nu-i}
 \left(\frac{\osc(u,2R)}{2^{\nu+1}}\right)^{\frac{p^-}{p^+}} = \left(\frac{\osc(u,2R)}{2^{\nu+1}R^{\frac{p^-}{p^+}}}\right)^{\frac{p^-}{p^+}} R^{\frac{(p^-)^2}{(p^+)^2}}
\leq \frac{\osc(u,2R)}{2^{\nu+1}R^{\frac{p^-}{p^+}}} R^{\frac{(p^-)^2}{(p^+)^2}}
= \frac{\osc(u,2R)}{2^{\nu+1}} R^{\frac{(p^-)^2}{(p^+)^2} - \frac{p^-}{p^+}}.
\end{align}
By Remark \ref{new-rem-4.5} we may fix $\nu \geq i_0$ large enough and $R$ small enough so that
\begin{align}\label{osc-and-nu-ii}
c\left(\frac{\tilde{C}}{\nu}\right)^{\frac{1+\beta}{p^+}}  R^{\frac{p^--p^+}{p^+}} R^{\frac{p^-}{p^+}
\left(\frac{p^--p^+}{p^+}\right)} \leq \frac12.
\end{align}
%since $\frac{p^-}{p^+}-1 = \frac{1}{p^+}(p^- - p^+)$ and $\frac{(p^-)^2}{(p^+)^2} - \frac{p^-}{p^+} = \frac{p^-}{(p^+)^2}(p^- - p^+)$, and quantities of the form $R^{p^- - p^+}$ are small for small $R$.
Combining \eqref{osc-and-nu-i} and \eqref{osc-and-nu-ii} we get that
\[
c\left(\frac{\tilde{C}}{\nu}\right)^{\frac{1+\beta}{p^+}}  R^{\frac{p^--p^+}{p^+}} \left(\frac{\osc(u,2R)}{2^{\nu+1}}\right)^{\frac{p^-}{p^+}} \leq \frac{\osc(u,2R)}{2^{\nu+2}},
\]
%%%%%%%%%%%%%%%%%%%%%%%%%%%%%%%%%
and therefore
\[
\osc(u,R/2) \leq \osc(u,2R) \left( 1 - \frac{1}{2^{i+1}} \right) + c R^{\frac{1}{p^+}\left(1-\frac{n}{t}\right)}.
\]

On the other hand, suppose that
\[
\frac{\osc(u,2R)}{2^{\nu+1}} > R^{\frac{p^-}{p^+}}.
\]
Then
\[
\osc(u,R/2) \leq 2^{\nu+1} R^{p^-/p^+}.
\]

In both cases we have that
\[
\osc(u,R/2) \leq \osc(u,2R) \left( 1 - \frac{1}{2^{i+1}} \right) + c 2^{\nu+1} R^{\frac{p^-}{p^+}} \max\left\{1,\frac{1}{R^{\frac{n}{tp^+}}}\right\}.
\]
We are in a position to apply Lemma 7.3 from \cite{Giusti}, cf. Lemma 4.6 in \cite{Chiado-Coscia}, with
\begin{align*}
f(t) &= \osc(u,2t), \\
\tau &= 1/4, \\
\sigma &= \log_\tau \left( 1 - \frac{1}{2^{\nu+2}} \right), \\
A &= c 2^{\nu+2} \quad\textrm{ and} \\
\alpha &< \min\left\{\sigma, \left(p^- - \frac{n}{t}\right)\frac{1}{p^+} \right\}.
\end{align*}
Consequently we get that for every $r < R \leq 1$
\[
\osc(u,2r) \leq c \left( \left(\frac{r}{R}\right)^\alpha \osc(u,2R) + Ar^\alpha \right),
\]
where $c = c(\tau,\sigma,\alpha)$. It follows that $u \in C^{0,\alpha}(\Omega)$.
\end{proof}

\section{Applications to PDEs}\label{Section4}

The purpose of this section is to apply
Theorem~\ref{new-Theorem-4-1} to a class of $(A, B)$-harmonic
equations, with variable exponent growth of $A$ and $B$, and to show
the H\"older continuity of their solutions. This goal is obtained by
proving that such solutions are quasiminimizers of an energy
functional $J_{\Om}$ under the general growth type \eqref{F-growth}, see
\eqref{defn-J}--\eqref{definition-of-z} and Theorem~\ref{AB-quasim}. The fact that, under our assumptions, the
coefficient $b$ may depend on $x\in \Om$ allows us to cover wider
classes of PDEs than those studied so far in the literature, cf.
Theorem 2.1 in Giusti--Giaquinta~\cite{Gia-Giu} and Theorem 2.2 in
Fan--Zhao~\cite{Fan-Z}. We illustrate our discussion with Examples 1
and 2. There we show the local H\"older continuity of weak solutions
to equations
\begin{align*}
-\div(|\nabla u|^{p(x)-2}\nabla u) &= V(x) |u|^{q(x)-2}u, \\
-\div (|\nabla u|^{p(x)-2}\nabla u) + |u|^{p(x)-2}u & = \lambda g(x) |u|^{\alpha(x)-2}u - h(x) |u|^{\beta(x)-2}u +K(x).
\end{align*}
Let us consider the following elliptic equation in a domain $\Om\subset \R^n$
\begin{align}\label{sol-equation}
-\div A(x,u,\nabla u) = B(x,u,\nabla u),
\end{align}
where $A, B:\Om\times\R\times\R^n\to \R$ satisfy the structural conditions
\begin{align}
A(x,u,\xi)\xi &\geq \mu |\xi|^{p(x)} - b(x)|u|^{\tilde{r}(x)} - f(x), \nonumber \\
|A(x,u,\xi)| &\leq \mu|\xi|^{p(x)-1} + b(x)|u|^{\tsigmax} + g(x), \label{AB-cond}\\
|B(x,u,\xi)| &\leq \mu|\xi|^{\ttaux} + b(x)|u|^{\tdeltax} + h(x).\nonumber
\end{align}
Here $\mu\geq 1$, $f \in L^t(\Om)$, $g \in L^{tp'(\cdot)}(\Om)$ and
$h \in L^{tr'(\cdot)}(\Om)$; $b$ will be considered momentarily. For
the exponents, we assume their boundedness and define them for $x\in
\Om$ as follows:
\begin{align}
p(x)&\leq r(x)\leq p^*(x), \nonumber \\
\tilde{r}(x) &= r(x)-\tilde{\epsilon} \nonumber \qquad \textrm{ for some } 0 < \tilde{\epsilon} < p^--1, \nonumber \\
\tilde{\sigma}(x) &= r(x)\frac{p(x)-1}{p(x)}, \qquad
\tilde{\tau}(x) = p(x)\left( 1 - \frac{1}{\tilde{r}(x)}\right), \qquad
\tilde{\delta}(x) = \tilde{r}(x) - 1. \label{AB-exp-cond}
\end{align}
%The assumption on $\tilde{\epsilon}$ guarantees that $\tilde{r} > 1$, and thus $\tilde{\tau}, \tilde{\delta} > 0$.
These exponents differ from those of Fan--Zhao~\cite[Theorem 2.2]{Fan-Z} by using $\tilde{r} = r-\tilde{\epsilon}$ instead of $r$. Therefore, our assumptions on the exponents are marginally stronger.

%As for the integrability of $b(x)$ in \eqref{AB-cond},
In order to use Theorem~\ref{new-Theorem-4-1} we want $\mathbb{B}
\in L^{\sigma}(\Om)$ and $z \in L^t(\Om)$ to hold for the auxiliary functions
$z$ and $\mathbb{B}$ defined below, see \eqref{definition-of-z}. The
first condition requires that $b$, $b^{\frac{\p}{\p-1}}$ and
$b^{\frac{\rdot}{\tilde{r}(\cdot)-1}}$ belong to $L^\sigmadot(\Om)$. The
condition which implies all three cases is
\[
b \in L^{\sigmadot \frac{\p}{\p-1}}(\Om) \qquad \textrm{ when } p \leq
\frac{r}{1+\tilde{\epsilon}},
\]
and
\[
b \in L^{\sigmadot \frac{\rdot}{\tilde{r}(\cdot)-1}}(\Om) \qquad \textrm{
when } p> \frac{r}{1+\tilde{\epsilon}}.
\]
Additionally, in order to have $z \in L^t(\Om)$ we require that
$\mathbb{B}^{\frac{p^*_-}{p^*_- - r_+}} \in L^t(\Om)$. Combining the
demands of $\mathbb{B}$ and $z$ results in the assumption $b \in
L^{s(\cdot)}(\Om)$, where
\begin{equation*}
s(x) =
\begin{cases}
& \frac{p(x)}{p(x)-1} \max\{\sigma(x), t \frac{p^*_-}{p^*_- - r_+} \} \qquad \textrm{ if } p(x) \leq  \frac{r(x)}{1+\tilde{\epsilon}},  \\
 & \frac{r(x)}{\tilde{r}(x)-1} \max\{\sigma(x), t \frac{p^*_-}{p^*_- - r_+} \} \qquad \textrm{ if } p(x) > \frac{r(x)}{1+\tilde{\epsilon}}.
\end{cases}
\end{equation*}

If $p > \frac{r}{1+\tilde{\epsilon}}$, then $\frac{r}{r-\tilde{\epsilon}-1} \leq \frac{p(1+\tilde{\epsilon})}{p-\tilde{\epsilon}-1}$. Moreover, it always holds that $\frac{p}{p-1} \leq \frac{p(1+\tilde{\epsilon})}{p-\tilde{\epsilon}-1}$. Thus, we can replace the above integrability condition on $s$ with the following one:
\begin{equation}\label{this-is-s}
s(x) = \frac{p(x)(1+\tilde{\epsilon})}{p(x)-\tilde{\epsilon}-1} \max\{\sigma(x), t \frac{p^*_-}{p^*_- - r_+} \}.
\end{equation}

\begin{rem}
We have made no assumption on the relative sizes of $\sigma$ and $t$. If we explicitly chose $t = n + \epsilon_c$ and $\sigma = \frac{p^*}{p^*-r} + \epsilon_c$ for some small $\epsilon_c > 0$, cf. assumptions \eqref{F-growth2}, we could write \eqref{this-is-s} as $b \in L^{s(\cdot)}(\Om)$ where $ s = (1+\tilde{\epsilon})(n+ \epsilon_c) \frac{p}{p-\tilde{\epsilon}-1} \frac{p^*_-}{p^*_- - r_+}$, and proceed with this assumption as below.
\end{rem}

We say that $u \in \Wpx(\Omega)$ is a weak solution of equation \eqref{sol-equation} in $\Omega$, if for all $\phi \in C^\infty_0(\Omega)$, it holds that
\[
\int_\Omega A(x,u,\nabla u)\nabla\phi\,dx = \int_\Omega  B(x,u,\nabla u)\phi\,dx.
\]
We then extend the pool of test functions to allow $\phi \in \Wpxzero(\Omega)$ by the density argument, see e.g. \cite[Chapter 9]{DHHR}.

Define an energy functional $J_{\Om}(u)$ as
\begin{equation}\label{defn-J}
J_{\Om}(u) := \int_{\Om} \left( |\nabla u|^{p(x)} + \mathbb{B}(x)|u|^{r(x)} + z(x) \right)\,dx,
\end{equation}
where
\begin{align}\label{definition-of-z}
\mathbb{B}(x)&:= c(p^+, p^-, r^+, r^-)\left( b(x) + b(x)^{\frac{p(x)}{p(x)-1}} + b(x)^{\frac{r(x)}{\tilde{r}(x)-1}} \right),\nonumber \\
z(x)&:= 1 + b(x) + f(x) + g(x)^{p'(x)} + h(x)^{r'(x)} + \mathbb{B}(x)+\mathbb{B}^{\frac{p^*_-}{p^*_- - r_+}}(x).
\end{align}
 By the above discussion on the integrability of coefficients and exponents in \eqref{AB-cond}, \eqref{AB-exp-cond} and \eqref{this-is-s}, we have that $B, z$ satisfy assumptions \eqref{F-growth} and \eqref{F-growth2}.
 
\begin{thm}\label{AB-quasim}
If $u$ is a weak solution of equation \eqref{sol-equation} under assumptions \eqref{AB-cond}--\eqref{this-is-s}, then $u$ is a local quasiminimizer of $J_{\Om}$.
\end{thm}

For the definition of local quasiminimizers we refer to the discussion following Definition~\ref{quasi-def}.

\begin{rem}
We would like to emphasize that the dependence of the coefficient $\mathbb{B}$ in $J_{\Om}$ on a point $x \in \Om$ requires extra attention in the estimates below, see \eqref{from-struct-conditions} and further discussion. In the previously studied results $\mathbb{B}$ is either a constant or belongs to $L^{\infty}$, cf. Fan-Zhao~\cite{Fan-Z}.
\end{rem}

By combining Theorem~\ref{AB-quasim} with the local H\"older continuity result in Theorem~\ref{new-Theorem-4-1} we immediately obtain the following observation.

\begin{cor}\label{AB-quasi-cor}
 Assume that \eqref{AB-cond}--\eqref{this-is-s} hold. Then weak solutions of \eqref{sol-equation} are locally
 H\"older continuous.
\end{cor}

\begin{proof}
 It is enough to check that functions $\mathbb{B}$ and $z$, as defined above, satisfy the assumptions of Theorem~\ref{new-Theorem-4-1}.
\end{proof}

Before proving Theorem~\ref{AB-quasim} let us illustrate our presentation with examples of PDEs covered by our result.
\begin{ex}
Let $u\in \Wpxloc(\Om)$ be a weak solution of
\[
-\div(|\nabla u|^{p(x)-2}\nabla u) = V(x) |u|^{q(x)-2}u
\]
for $\p$ as in assumptions \eqref{p-assump} and a bounded continuous exponent $q$ such that $p-\tilde{\epsilon} \leq q \leq p^*-\tilde{\epsilon}$ in $\Om$ for $\tilde{\epsilon} < p^--1$; also let $V\in L^{\sigma(\cdot)}(\Om)$.
 %Computations reveal that $V\in L^{\sigma(\cdot)}$ for $\min\{p, \sqrt{p^*(1+\epsilon))}\}\leq r\leq p^*$.
 Then $u$ satisfies assumptions of Theorem~\ref{AB-quasim} and, thus, by Corollary~\ref{AB-quasi-cor}, $u$ is locally H\"older continuous. Indeed, $A(x,u,\nabla u) = |\nabla u|^{p(x)-2}\nabla u$, whereas $B(x,u,\nabla u) = V(x) |u|^{q(x)-2}u$, and $A$ and $B$ satisfy \eqref{AB-cond} and \eqref{AB-exp-cond} for $\mu=1$, $f\equiv g\equiv h\equiv 0$ and  $\tilde{\delta} = q-1$. The latter implies that $r = q+\tilde{\epsilon}$ and so, by \eqref{AB-exp-cond}, one needs to assume that $p\leq q+\tilde{\epsilon}\leq p^*$.
\end{ex}

\begin{ex}
In Aouaoui~\cite{aou}, the following eigenvalue problem is studied in the context of existence and multiplicity of solutions:
\[
-\div (|\nabla u|^{p(x)-2}\nabla u) + |u|^{p(x)-2}u = \lambda g(x) |u|^{\alpha(x)-2}u - h(x) |u|^{\beta(x)-2}u +K(x),
\]
where $p$, $\alpha$ and $\beta$ are variable exponents, and $h, g$ and $K$ are positive functions. Moreover, \cite[Theorem 1.1]{aou} shows the existence of a solution for some $\lambda>0$ provided that $2 < \alpha(x) < p^*(x)$ in the set of points $x$ such that $\alpha(x) \geq \beta(x)$. Let us find conditions implying that $u$ is a quasiminimizer of an energy subject to Theorem~\ref{AB-quasim} and, thus, $u$ is H\"older continuous.

 Clearly, $A(x, u, \nabla u) = |\nabla u|^{p(x)-2}\nabla u$ and $B(x, u) = - |u|^{p(x)-2}u + \lambda g(x) |u|^{\alpha(x)-2}u - h(x) |u|^{\beta(x)-2}u + K(x)$. We estimate
\begin{align*}
|B(x, u)| &\leq \lambda g(x) |u|^{\alpha(x)-1} + h(x) |u|^{\beta(x)-1} + |u|^{p(x)-1} + K(x) \\
%&\leq \left(\lambda g(x) + h(x) +1 \right) \left( 3 + |u|^{\max\{\alpha(x)-1,\beta(x)-1,p(x)-1\}} \right) + K \\
&\leq \left(\lambda g(x) + h(x) +1 \right) |u|^{\max\{\alpha(x)-1,\beta(x)-1,p(x)-1\}} + \left(\lambda g(x) + h(x) + K(x)+1 \right).
\end{align*}
By discussion in \cite{aou} we have that exponent $P(x) = \max\{\alpha(x)-1,\beta(x)-1,p(x)-1\}$ is bounded continuous and $P^{-}>0$. One immediately checks that the differential operator $A$ satisfies \eqref{AB-cond} and \eqref{AB-exp-cond} for any $\tilde{\sigma}$, $b\geq 0$ and $g\geq 0$, while $|B(x,u)| \leq b(x)|u|^{\tilde{\delta}} + k(x)$ with $b(x) = \lambda g(x) + h(x)+1$, $k(x) = \lambda g(x) + h(x) + K(x)+1$ and
\[
\tilde{\delta} = P(x)= \max\{\alpha(x),\beta(x),p(x)\}-1.
\]
 We require $\tilde{\delta} = \tilde{r}-1 = r - \tilde{\epsilon} - 1$. This implies that assumptions of Theorem~\ref{AB-quasim} are satisfied upon defining $r(x):= \max\{\alpha(x),\beta(x),p(x)\} + \tilde{\epsilon}$ and assuming that $b\in L^{s(\cdot)}, k\in L^t$ and $r=\max\{\alpha(x),\beta(x),p(x)\} + \tilde{\epsilon} \leq p^*(x)$.
\end{ex}

\begin{proof}[Proof of Theorem~\ref{AB-quasim}]
By the continuity of variable exponents $\p$ and $\rdot$ in $\Om$, for any $y\in \Om$ we may find its small neighborhood $\Om_{y}$ such that $r_+ \leq p^*_-$ in $\Om_{y}$. To this end we localize discussion by choosing $\Om'\Subset \Om_{y}$.

Let $v \in \Wpxloc(\Omega)$ be such that $u-v \in \Wpxzero(\Omega')$, with $S:= \supp (u-v) \subset\subset \Omega'$.

Since $u$ is a weak solution, we have
\[
\int_{\Om'} A(x,u,\nabla u)\nabla u = \int_{\Om'} A(x,u,\nabla u)\nabla v + \int_{\Om'} B(x,u,\nabla u)(u-v).
\]
Using the structural conditions \eqref{AB-cond}, we get
\begin{align}\label{from-struct-conditions}
\nonumber
\mu \int |\nabla u|^{p(x)} &\leq
(I) \int b(x)|u|^{\tilde{r}(x)} +
(1) \int f(x) \\ &+
(2) \mu\int |\nabla u|^{p(x)-1}|\nabla v| +
(II) \int b(x)|u|^{\tsigmax}|\nabla v| +
(3) \int g(x)|\nabla v| \\ &+
\nonumber
(4) \mu\int|\nabla u|^{\ttaux}|u-v| +
(III) \int b(x)|u|^{\tdeltax}|u-v| +
(5) \int h(x)|u-v|.
\end{align}
By using Arabic and Roman numerals we distinguish two types of integrals, depending whether they contain function $b$ or not. In the previously studied results $b$ is either a constant or belongs to $L^{\infty}$, while for us integrals (I)--(III) require additional effort.

For integrals (2), (3) and (5) we make the following simple estimates using Young's inequality:
\begin{align}\label{evals-1-through-4}
\mu\int |\nabla u|^{p(x)-1}|\nabla v| &\leq \epsilon_0 \mu\int |\nabla u|^{p(x)} +  \mu \int \epsilon_0^{1-p(x)} |\nabla v|^{p(x)}, \\
\nonumber \int g(x)|\nabla v| &\leq \int |\nabla v|^{p(x)} + \int g(x)^{p'(x)}, \\
\nonumber \int h(x)|u-v| &\leq \int |u-v|^{r(x)} + \int h(x)^{r'(x)}.
\end{align}
\noindent Similarly, we estimate integral (4) as follows
\begin{align}\label{eval-II}
 \int|\nabla u|^{\tilde{\tau}(x)}|u-v| &\leq \epsilon_3 \int|\nabla u|^{p(x)} + \int \epsilon_3^{-\frac{\px}{p(x) - \tilde{\tau}(x)}} |u-v|^{\frac{p(x)}{p(x) - \tilde{\tau}(x)}} \nonumber \\
&\leq \epsilon_3 \int|\nabla u|^{p(x)} + \int \epsilon_3^{-\frac{\px}{p(x) - \tilde{\tau}(x)}} |u-v|^{\tilde{r}(x)}. \end{align}
Note that $p - \tilde{\tau} > 0$ in $\Om$ for our choice of $\tilde{\tau}$.

Let us estimate the integrals (I) through (III) next. For integral (I), we take $\alpha = 1 - \frac{\tilde{r}(x)}{r(x)}$ and $\beta$ so that $\alpha + \beta = 1$. Then $\beta \frac{r(x)}{\tilde{r}(x)} = 1$, and by using the Young inequality for $\epsilon_1\in(0,1)$ we get
\begin{align}\label{eval-IV}
\nonumber \int b(x)|u|^{\tilde{r}(x)} &= \int b(x)^\alpha b(x)^\beta |u|^{\tilde{r}(x)} \\
&\leq \int \epsilon_1^{-\frac{r(x)}{r(x)-\tilde{r}(x)}} b(x)^{\alpha \frac{r(x)}{r(x) - \tilde{r}(x)}} + \int \epsilon_1 b(x)^{\beta \frac{r(x)}{\tilde{r}(x)}} |u|^{r(x)} = \int \epsilon_1^{-\frac{r(x)}{r(x)-\tilde{r}(x)}} b(x) + \int \epsilon_1 b(x) |u|^{r(x)}.
\end{align}

\noindent For (II), we note that
\begin{align}\label{eval-I}
\int b(x)|u|^{\tilde{\sigma}(x)}|\nabla v| &\leq \int \epsilon_2^{p'(x)} b(x)^{\frac{p(x)}{p(x)-1}} |u|^{r(x)} + \int \epsilon_2^{-p(x)} |\nabla v|^{p(x)}
\end{align}
by our choice of $\tilde{\sigma}$, see \eqref{AB-exp-cond}.

\noindent Finally, we estimate integral (IV)
\begin{align}\label{eval-III}
\int b(x)|u|^{\tilde{\delta}(x)}|u-v| &= \int b(x)|u|^{\tilde{r}(x) - 1}|u-v| \leq \int b^{\frac{\tilde{r}(x)}{\tilde{r}(x)-1}} |u|^{\tilde{r}(x)} + \int |u-v|^{\tilde{r}(x)},\nonumber \\
&\leq \int \epsilon_4^{-\frac{r(x)}{r(x)-\tilde{r}(x)}} + \int \epsilon_4 b(x)^{\frac{\tilde{r}(x)}{\tilde{r}(x)-1} \frac{r(x)}{\tilde{r}(x)}} |u|^{r(x)}+\int |u-v|^{\tilde{r}(x)}.
\end{align}

%so we get
%\begin{align}\label{eval-III}
%\int b(x)|u|^{\tilde{\delta}(x)}|u-v|
%\leq \int \epsilon_3^{-\frac{\tilde{r}(x)}{r(x)-\tilde{r}(x)}} + \int \epsilon_3 b(x)^{\frac{r(x)}{\tilde{r}(x)-1}} |u|^{r(x)} + \int |u-v|^{\tilde{r}(x)}
%\end{align}

Upon collecting estimates \eqref{evals-1-through-4}--\eqref{eval-III}, we arrive at the following inequality
\begin{align}
& \mu \int |\nabla u|^{p(x)} \leq \int \epsilon_1^{-\frac{r(x)}{r(x)-\tilde{r}(x)}} b(x) + \int \epsilon_1 b(x) |u|^{r(x)}+ \int f(x) \nonumber \\
&+ \epsilon_0 \mu\int |\nabla u|^{p(x)} +  \mu\int \epsilon_0^{1-p(x)} |\nabla v|^{p(x)}
+ \int \epsilon_2^{p'(x)} b(x)^{\frac{p(x)}{p(x)-1}} |u|^{r(x)} + \int \epsilon_2^{-p(x)} |\nabla v|^{p(x)} \nonumber \\
&+ \int g(x)^{p'(x)} + \int |\nabla v|^{p(x)} \label{AB-collecting1}\\
&+ \epsilon_3 \mu \int|\nabla u|^{p(x)} + \mu \int \epsilon_3^{-\frac{\px}{p(x) - \tilde{\tau}(x)}} |u-v|^{\tilde{r}(x)} + \int \epsilon_4^{-\frac{r(x)}{r(x)-\tilde{r}(x)}} + \int \epsilon_4 b(x)^{\frac{r(x)}{\tilde{r}(x)-1}} |u|^{r(x)} + \int |u-v|^{\tilde{r}(x)} \nonumber \\
&+ \int |u-v|^{r(x)} + \int h(x)^{r'(x)}. \nonumber
\end{align}
We rearrange terms in estimate \eqref{AB-collecting1} to obtain
\begin{align}
\mu \int |\nabla u|^{p(x)} &\leq  \mu (\epsilon_0 + \epsilon_3) \int |\nabla u|^{p(x)}
+ \int \left( \epsilon_0^{1-p(x)}\mu + \epsilon_2^{-p(x)} + 1\right) |\nabla v|^{p(x)} \label{epsilon-corral}\\
&+ \int \left( \mu \epsilon_3^{-\frac{\px}{p(x) - \tilde{\tau}(x)}} + 1 \right) |u-v|^{\tilde{r}(x)}
+ \int |u-v|^{r(x)} \label{u-v-corral} \\
%
%CORRAL FOR tilde-b-ur
\label{tilde-b-ur-corral}
&+ \int \left(\epsilon_1 b(x) + \epsilon_2^{p'(x)} b(x)^{\frac{p(x)}{p(x)-1}} + \epsilon_4 b(x)^{\frac{r(x)}{\tilde{r}(x)-1}} \right)
|u|^{r(x)} \\
&+ \int \epsilon_1^{-\frac{r(x)}{r(x)-\tilde{r}(x)}} b(x)
+ f(x) + g(x)^{p'(x)} + \epsilon_4^{-\frac{r(x)}{r(x)-\tilde{r}(x)}}+ h(x)^{r'(x)}. \nonumber
\end{align}
We denote the integrand on the last line by $\tilde{g}$. Upon choosing $\epsilon_0$ and $\epsilon_3$ so that $\mu(\epsilon_0 + \epsilon_3)<1$, we include the $|\nabla u|$-term in the left-hand side of \eqref{epsilon-corral}. For line \eqref{tilde-b-ur-corral}, we define
\[
\mathbb{B}(x):= \max\{\epsilon_2^{p'_-},\epsilon_1,\epsilon_4\} \left( b(x) + b(x)^{\frac{p(x)}{p(x)-1}} + b(x)^{\frac{r(x)}{\tilde{r}(x)-1}} \right).
\]
Then we estimate expression in \eqref{tilde-b-ur-corral} by $\int \mathbb{B}(x) |u|^{r(x)}$, and proceed as follows. First,
\begin{align}
\int \mathbb{B}(x) |u|^{r(x)} &\leq  2^{r_+}\int \mathbb{B}(x) \left( |u-v|^{r(x)} + |v|^{r(x)} \right)
\leq c(r^+)\int \mathbb{B}(x) \left( |u-v|^{r_+} + 1 + |v|^{r(x)} \right) \nonumber\\
 &= c(r^+)\left(\int \mathbb{B}(x) |u-v|^{r_+} + \int \mathbb{B}(x) |v|^{r(x)} + \int \mathbb{B}(x) \right).
 \label{destroy-b-ur}
\end{align}
We include the last term in \eqref{destroy-b-ur} in $\tilde{g}$. For the first term in \eqref{destroy-b-ur} we again use the Young inequality:
\begin{align}\label{destroy-u-v-r-plus}
 c(r_+) \int \mathbb{B}(x) |u-v|^{r_+} \leq \epsilon_5c(r_+) \int |u-v|^{p^*_-} + \epsilon_5^{- \frac{r_+}{p^*_- - r_+}} c(r_+) \int \mathbb{B}(x)^{\frac{p^*_-}{p^*_- - r_+}},
\end{align}
since $r_+ \leq p^*_-$ in sufficiently small balls, see discussion in the beginning of the proof.
As for the first right-hand side integral in \eqref{destroy-u-v-r-plus}, by the Sobolev embedding theorem, we have
\begin{align}
\int |u-v|^{p^*_-} &\leq C(C_{Sob}, p^*_-) \left( \int |\nabla u|^{p_-} + |\nabla v|^{p_-} \right)^{\frac{p^*_-}{p_-}} \label{AB-Sobemb}\\
&\leq C(C_{Sob}, p^*_-, p^-) \left( \int |\nabla u|^{p(x)} + |\nabla v|^{p(x)} + 2 \right)^{\frac{p^*_-}{p_-}} \nonumber
\end{align}
\begin{align}
&= C(C_{Sob}, p^*_-, p^-) \left( \int |\nabla u|^{p(x)} + |\nabla v|^{p(x)} + 2 \right)^{\frac{p^*_-}{p_-}- 1} \left( \int |\nabla u|^{p(x)} + |\nabla v|^{p(x)} + 2 \right) \nonumber \\
&= C \int (|\nabla u|^{p(x)} + \mathbb{B}(x)|u|^{r(x)} + 2).\nonumber
\end{align}
Here $C=(C_{Sob}, p^*_-, p^-, \|u\|_{\Wpxloc(\Om')})$. In the last step we also use the observation that
\[
\int |\nabla v|^{p(x)} < \int |\nabla u|^{p(x)} + \mathbb{B}(x)|u|^{r(x)},
\]
since otherwise $u$ is a local minimum of $J$ and so, in particular, a quasiminimum. For details see the similar discussion in the proof of Lemma~\ref{new-lem-2.4}, inequality \eqref{Lem2.4:ineq-trick}. We use the above estimate in \eqref{destroy-u-v-r-plus} and upon choosing $0<\epsilon_5<1$ so that $\epsilon_5 c(r^+)C<1$, we include the $|\nabla u|$-term in the left-hand side of \eqref{epsilon-corral}.

The similar argument applies to expressions in \eqref{u-v-corral}, since $\tilde{r} < r$, and hence
\[
|u-v|^{r(x)} \leq |u-v|^{r_+} + 1 \quad \textrm{ and } \quad |u-v|^{\tilde{r}(x)} \leq |u-v|^{r_+} + 1.
\]
From this we can proceed as in \eqref{destroy-u-v-r-plus} and \eqref{AB-Sobemb}.

We collect together the above estimates and apply them in \eqref{epsilon-corral} to arrive at the following inequality
\begin{align}
\mu \int |\nabla u|^{p(x)} \leq C\left( \int |\nabla v|^{p(x)}+ \epsilon \int \mathbb{B}(x) |u|^{r(x)} + \int \mathbb{B}(x)|v|^{r(x)} + \int (\tilde{g} + \mathbb{B}(x)+\mathbb{B}(x)^{\frac{p^*_-}{p^*_- - r_+}})\right).\label{AB-epsilon-collect2}
\end{align}
Here $\epsilon \in (0, 1)$ while $C$ depends on parameters of constants in the preceding estimates.

Denote $z(x):=\tilde{g}(x) + \mathbb{B}(x)+\mathbb{B}^{\frac{p^*_-}{p^*_- - r_+}}(x)$.

We add $\int \mathbb{B}(x) |u|^{r(x)} + \int z(x)$ to the both sides of \eqref{AB-epsilon-collect2}, to get that
\begin{align}\label{getting-there}
\mu \int |\nabla u|^{p(x)} + (1-C\epsilon) \int \mathbb{B}(x) |u|^{r(x)} + \int z(x) &\leq C \Big(\int |\nabla v|^{p(x)}    +\int \mathbb{B}(x)|v|^{r(x)}+ \int z(x) \\
&+ \int \mathbb{B}(x)|u|^{r(x)} \Big). \nonumber
\end{align}
Arguing as in \eqref{destroy-b-ur} through \eqref{AB-Sobemb} we have
\begin{align*}
\int \mathbb{B}(x) |u|^{r(x)}
\leq
c(r^+)\left( \int \mathbb{B}(x) + \int \mathbb{B}(x)|v|^{r(x)}\right) + \int \epsilon_6 |u-v|^{p^*_-} + \int \mathbb{B}(x)^{\frac{p^*_-}{p^*_- - r_+}} \epsilon_6^{-\frac{r_+}{p^*_- - r_+}},
\end{align*}
and
\[
\int \epsilon_6 |u-v|^{p^*_-} \leq \epsilon_6 C\left( \int |\nabla u|^{p(x)} + \int |\nabla v|^{p(x)} + \int 2 \right).
\]
Thus, choosing $\epsilon_6$ small enough and the appropriate value of $\epsilon$, \eqref{getting-there} becomes
\begin{align*}
 \int \mu |\nabla u|^{p(x)} +  \mathbb{B}(x)|u|^{r(x)} + 2 z(x) & \leq
 C \int \left( |\nabla v|^{p(x)} + \mathbb{B}(x)|v|^{r(x)}+ z(x)\right)\\
 &+C\int\left(\mathbb{B}(x)+\mathbb{B}(x)^{\frac{p^*_-}{p^*_- - r_+}}+2\right).
\end{align*}
We apply the definition of $k$ and increase the value of $C$ if necessary and obtain that
\[
 J_{\Om'}(u)\leq C J_{\Om'}(v).
\]
This completes the proof of Theorem~\ref{AB-quasim}.
%This is, dividing with $\tilde{C}$,
%\begin{align*}
%& \mu \int |\nabla u|^{p(x)} + \tilde{\tilde{C}} \int \tilde{b} |u|^{r(x)} + \tilde{\tilde{\tilde{C}}} \int \tilde{g} \\
%&\leq \tilde{C} \left( \mu \int |\nabla v|^{p(x)}
%+ \tilde{\tilde{C}} \int \tilde{b} |v|^{r(x)} + \tilde{\tilde{\tilde{C}}} \int \tilde{g} \right).
%\end{align*}
\end{proof}

\section{Appendix}

The purpose of this section is to prove Lemmas~\ref{new-lem-2.4} and \ref{new-ren-2.5}, see Preliminaries for the formulation of the lemmas, also Remark \ref{Remark-K-and-norm}.

\begin{proof}[Proof of Lemma~\ref{new-lem-2.4}]
 We define an auxiliary energy functional
 \begin{equation*}
 \widetilde \calFone(u):=\calFone_{, \Om'}(u)-\int_{\Om'} h(x)\,dx=\int_{\Om'}(\mu|\nabla u|^{\px}+b(x)|u|^{\rx}+g(x))\,dx.
 \end{equation*}
 Let $\Om''\Subset \Om'$ be a subdomain of $\Om'$. In the proof we will need estimates for the integrand terms of $\calFone$. We start with the following inequality for a test function $\phi\in \Wpxzero(\Om'')$ and some $\epsilon\in (0,1)$:
  \begin{align}
   b(x)|u|^{\rx}&=b(x)|u+\phi-\phi|^{\rx} \leq 2^{r^+}\left(\epsilon^{\frac{\rx}{\pstarx}}|\phi|^{\rx} b(x)
   \epsilon^{-\frac{\rx}{\pstarx}}+b(x)|u+\phi|^{\rx} \right) \nonumber \\
   &\leq 2^{r^+}\left(\epsilon |\phi|^{\pstarx}+ \epsilon^{-\frac{\pstarx}{\pstarx-\rx}}b(x)^{\frac{\pstarx}{\pstarx-\rx}}+b(x)|u+\phi|^{\rx} \right).\label{Lem2.4:est1}
  \end{align}
 In the last step above we used the Young inequality together with the assumption $r \leq p^{*}$ in $\Om$.

 We can suppose that
 \begin{equation}\label{Lem2.4:ineq-trick}
 \int \limits_{\supp \phi} \mu |\nabla (u+\phi)|^{\px}< \int \limits_{\supp \phi} \mu |\nabla u|^{\px}+b(x) |u|^{\rx},
 \end{equation}
 since otherwise it holds that
 \begin{align*}
  \widetilde \calFone(u)&=\int \limits_{\supp \phi} \mu|\nabla u|^{\px}+ b(x) |u|^{\rx}+g(x) \leq \int \limits_{\supp \phi} \mu |\nabla (u+\phi)|^{\px}+g(x) \\
  &\leq \int \limits_{\supp \phi} \mu |\nabla (u+\phi)|^{\px}+b(x)|u+\phi|^{\rx}+g(x)=\widetilde \calFone(u+\phi).
 \end{align*}
Hence $u$ is a minimizer of $\widetilde \calFone$ and so also a minimizer of $\calFone$. In particular, $u$ is a quasiminimizer of $\calFone$.

We will now estimate the term containing $|\phi|^{\pstar}$ in \eqref{Lem2.4:est1}.
By the Sobolev embedding theorem (cf. Diening--Harjulehto--H{\"a}st{\"o}--R{\r u}{\v z}i{\v c}ka~\cite[Theorem 8.3.1]{DHHR}) and the unit ball property defined in \eqref{unit-ball-prop} we have that
\begin{align*}
 &\int \limits_{\supp \phi} |\phi|^{\pstar} \\
 &\,\,\,\,\,\,\,\,\leq\max\{\|\phi\|^{\pstarm}_{L^{p^*(\cdot)}}, \|\phi\|^{\pstarp}_{L^{p^*(\cdot)}}\}\leq c\max\{\|\nabla \phi\|^{\pstarm}_{L^{\p}}, \|\nabla \phi\|^{\pstarp}_{L^{\p}}\} \\
 &\,\,\,\,\,\,\,\,\leq c\max\Big\{\max\{(\!\!\int \limits_{\supp \phi} |\nabla \phi|^{\p})^{\frac{\pstarm}{p^{-}}}, (\!\!\int \limits_{\supp \phi} |\nabla \phi|^{\p})^{\frac{\pstarm}{p^{+}}} \},
 \max\{(\!\!\int \limits_{\supp \phi} |\nabla \phi|^{\p})^{\frac{\pstarp}{p^{-}}}, (\!\!\int \limits_{\supp \phi} |\nabla \phi|^{\p})^{\frac{\pstarp}{p^{+}}} \}\Big\}.
\end{align*}
 By considering two cases: $\int_{\supp \phi} |\nabla \phi|^{\p}\leq (>)1$ we conclude that
 \begin{align}
  \int \limits_{\supp \phi} |\phi|^{\pstarx} & \leq c \max\Big\{\Big(\,\int \limits_{\supp \phi} |\nabla \phi|^{\p}\Big)^{\frac{p^{+}}{p^{-}}\frac{n}{n-p^{+}}}, \Big(\,\int \limits_{\supp \phi} |\nabla \phi|^{\p}\Big)^{\frac{n}{n-p^{-}}} \Big\} \nonumber \\
  &\leq c \bigg(\int \limits_{\supp \phi} |\nabla \phi|^{\p} \bigg) \max\Big\{\Big(\,\int \limits_{\supp \phi} |\nabla \phi|^{\p}\Big)^{\frac{p^{+}}{p^{-}}\frac{n}{n-p^{+}}-1}, \Big(\,\int \limits_{\supp \phi} |\nabla \phi|^{\p}\Big)^{\frac{n}{n-p^{-}}-1} \Big\}. \label{Lem2.4:est2}
 \end{align}
 Note that all powers of $|\nabla \phi|$-modulars on the right hand-side are strictly positive.
 We again appeal to observation \eqref{Lem2.4:ineq-trick} and are, therefore, able to provide the following estimate
 \begin{align}
  \int \limits_{\supp \phi} |\nabla \phi|^{\p} &\leq 2^{p^+} \int \limits_{\supp \phi} (|\nabla u|^{\p} + |\nabla (u+\phi)|^{\p}) \nonumber \\
  & \leq \frac{2^{p^+}}{\mu} \int \limits_{\supp \phi} (2\mu|\nabla u|^{\p}+b(x)|u|^{\rdot}) \nonumber \\
  & \leq c(\mu, p^+, r^+, \|u\|_{\Wpxloc(\Om)}, \|b\|_{L^{\sigma(\cdot)}}, |\Om|). \label{Lem2.4:est3}
 \end{align}
  In the last step we notice that since $p\leq r$, then $\frac{r\sigma}{\sigma-1}>p$ and thus by the variable exponent H\"older inequality we estimate $b|u|^{\rdot}$ in terms of $|\Om|$, $\|u\|_{L^{\p}_{loc}(\Om)}$ and $\|b\|_{L^{\sigmadot(\Om)}}$ for $\sigma > \frac{\pstar}{\pstar-r}$, see Remark 2 in Toivanen~\cite{Toivanen-2}.
 By combining \eqref{Lem2.4:est2} and \eqref{Lem2.4:est3} we obtain
 \begin{equation}\label{Lem2.4:est-phi}
 \int \limits_{\supp \phi} |\phi|^{\pstar} \leq C \int \limits_{\supp \phi} |\nabla \phi|^{\p},
 \end{equation}
 Constant $C$ depends on the variable exponent Sobolev embedding constant from \eqref{Lem2.4:est2} and parameters in $c$ from \eqref{Lem2.4:est3}.

 We are in a position to complete the proof of the lemma. By the quasiminimizing property of $u$, growth assumption \ref{F-growth}, the definition of $\calFone$ and inequalities \eqref{Lem2.4:est1}, \eqref{Lem2.4:est3} and \eqref{Lem2.4:est-phi}, we have that
 \begin{align}
 &\int \limits_{\supp \phi} \mu|\nabla u|^{\px}+ b(x) |u|^{\rx}+g(x) \leq K \widetilde \calFone(u+\phi) + 2 \int \limits_{\supp \phi} (b(x)|u|^{\rx}+g(x)) \nonumber \\
  & \leq K \widetilde \calFone(u+\phi) + \int \limits_{\supp \phi} 2g(x)+ 2^{r^+}\Big(\epsilon \int \limits_{\supp \phi} |\phi|^{\pstar}+ \int \limits_{\supp \phi} \epsilon^{-\frac{\pstar}{\pstar-r}}b(x)^{\frac{\pstar}{\pstar-r}}+\int \limits_{\supp \phi} b(x)|u+\phi|^{\rx} \Big) \nonumber \\
  & \leq K \widetilde \calFone(u+\phi) + \int \limits_{\supp \phi} 2g(x)+ 2^{r^++p^+}\epsilon C \Big(\int \limits_{\supp \phi} |\nabla u|^{\p}+ \int \limits_{\supp \phi} |\nabla (u+\phi)|^{\p} \Big) \nonumber \\
  &+ 2^{r^+}\int \limits_{\supp \phi} \epsilon^{-\frac{\pstar}{\pstar-\rdot}}b(x)^{\frac{\pstar}{\pstar-r}}+2^{r^+} \int \limits_{\supp \phi} b(x)|u+\phi|^{\rx}.\label{Lem2.4:est4}
 \end{align}
 Next, we choose $\epsilon$ so that $2^{r^++p^+}\epsilon C =\frac{\mu}{2}$. Furthermore, since $\sigma > \frac{\pstar}{\pstar-r}$, we may apply the variable exponent H\"older inequality and then the unit ball property to obtain that
 \begin{equation*}
  2^{r^+}\int \limits_{\supp \phi} \epsilon^{-\frac{\pstar}{\pstar-\rdot}}b(x)^{\frac{\pstar}{\pstar-\rdot}}\leq c,
 \end{equation*}
 where $c$ depends on $|\Om|$, $\|b\|_{L^{\sigmadot(\Om)}}$ and the choice of $\epsilon$. Then, we may include the $|\nabla u|^{\p}$-integral on the right-hand side of \eqref{Lem2.4:est4} in the left-hand side of \eqref{Lem2.4:est4}. In a consequence we arrive at the following:
 \begin{align}
  &\int \limits_{\supp \phi} \frac{\mu}{2}|\nabla u|^{\px}+ b(x) |u|^{\rx}+g(x) \leq K \widetilde \calFone(u+\phi) \nonumber \\
  &+\max\{2^{r^+}, \frac{\mu}{2}\}\Big( \int \limits_{\supp \phi} |\nabla (u+\phi)|^{\px}+b(x)|u+\phi|^{\rx}+g(x) \Big)+ \frac{2^{r^+}}{\epsilon^{(\frac{\pstar}{\pstar-\rdot})_{-}}}\int \limits_{\supp \phi}b(x)^{\frac{\pstarx}{\pstarx-\rx}}. \label{Lem2.4:est5}
 \end{align}
 Finally, from \eqref{Lem2.4:est5} we infer that
 \begin{align}
  & \frac12 \int \limits_{\supp \phi} \mu|\nabla u|^{\px}+ b(x) |u|^{\rx}+g(x)+h(x) \leq
  \frac12 \int \limits_{\supp \phi} \mu|\nabla u|^{\px}+ 2b(x) |u|^{\rx}+2g(x)+h(x) \nonumber \\
  &\leq K \widetilde \calFone(u+\phi) + c\Big( \calFone(u+\phi)+ \int \limits_{\supp \phi}b(x)^{\frac{\pstarx}{\pstarx-\rx}}\Big)+ \int \limits_{\supp \phi} \frac{h(x)}{2} \nonumber \\
  & \leq (K+c) \calFone(u+\phi) + \int\limits_{\Om''}\left(c b(x)^{\frac{\pstarx}{\pstarx-\rx}}+h(x)(\frac12-K-c)\right), \label{Lem2.4:est6}
 \end{align}
 where $c$ depends on $p^+, p^-, n, r^+, \mu$ and $\|u\|_{\Wpxloc}, \|b\|_{L^{\sigma(\cdot)}},\|g\|_{L^{t}}$. Since we defined $h:=\frac{c}{K+c-\frac12}b^{\frac{\pstar}{\pstar-\rdot}}$ we obtain that the last integral on the right-hand side of \eqref{Lem2.4:est6} is zero and so \eqref{Lem2.4:est6} reads
 \begin{equation*}
  \calFone(u)\leq 2(K+c) \calFone(u+\phi).
 \end{equation*}
 Thus, the proof of Lemma~\ref{new-lem-2.4} is completed.
\end{proof}

\begin{proof}[Proof of Lemma~\ref{new-ren-2.5}]
 We first show that under the assumptions of the lemma, $-u$ is a bounded minimizer of $\mathcal{F}_1$. Indeed,
 observe that $\calFone(-u)=\calFone(u)$, while for any $\phi\in \Wpxzero(\Om)$ it holds that
 \begin{align*}
  \calFone(-u+\phi)&=\int_{\Om}(\mu|-\nabla (u-\phi)|^{\px}+b(x)|-(u-\phi)|^{\rx}+g(x))dx=\calFone(u-\phi).
 \end{align*}
 Hence, the quasiminimizing property for $\calFone(-u)$ follows immediately from the corresponding property for $\calF$ applied with $-\phi$.

 Next, we show that under the assumptions of the lemma $u-k$ is the $K_1$-quasiminimizer of the energy $\calFtwo_{, \Om'}$. Let $\Om''\Subset \Om'$ and $\phi\in \Wpxzero(\Om'')$ be a test function. Then
 \begin{align*}
  &\calFtwo(u-k)= \int_{\Om}(\mu|\nabla (u-k)|^{\px}+b(x)|u-k|^{\rx}+g(x)+\tilde h(x))\\
&\leq \!\int \limits_{\supp \phi}( \mu |\nabla u|^{\px}+g(x))+\int \limits_{\supp \phi}\!\tilde h(x)+\int \limits_{\supp \phi} (2^{r(x)} b(x)|u|^{r(x)}+ 2^{r(x)} b(x)|k|^{r(x)}) \\
  & \leq K \int \limits_{\supp \phi}( \mu |\nabla (u+\phi)|^{\px}+b(x)|u+\phi-k+k|^{r(x)}+g(x)+h(x)) + \int \limits_{\supp \phi}\!\tilde h(x) \\
  &+ 2^{r+}K\int \limits_{\supp \phi} (\mu|\nabla (u+\phi)|^{\px}+b(x)|u+\phi-k+k|^{\rx}+g(x)+h(x))+\int \limits_{\supp \phi} 2^{r(x)} b(x)|k|^{r(x)} \\
  &\leq 2^{r^+}(2^{r+}+1)K\int \limits_{\supp \phi} (\mu|\nabla (u-k+\phi)|^{\px}+b(x)|u+\phi-k|^{\rx}+g(x)+h(x))\\
  &+ 2^{r^+}[(2^{r^+}+1)K+1]\int \limits_{\supp \phi} b(x)|k|^{r(x)}-\left(2^{r^+}(2^{r+}+1)K-1\right)\int \limits_{\supp \phi}\!\tilde h(x).
 \end{align*}
 We use the fact that $|k|^{\rdot}\leq (1+\sup_{\Om'}|u|)^{r^+}$ together with the definition of function $\tilde h$ to conclude the above estimations:
 \begin{align*}
& \calFtwo(u-k) \leq K_1 \calFtwo (u-k+\phi) \\
  & +2^{r^+}[(2^{r^+}+1)K+1](1+\sup_{\Om'}|u|)^{r^+}\int \limits_{\supp \phi} \left(b(x)-\frac{2^{r+}(2^{r+}+1)K-1}
      {2^{r+}[(2^{r^+}+1)K+1](1+\sup_{\Om'}|u|)^{r^+}} \tilde h(x)\right) \\
      &=K_1 \calFtwo (u-k+\phi).
 \end{align*}
 Thus, $u-k$ is a $K_1$-quasiminimizer of $\calFtwo$ and the proof of Lemma~\ref{new-ren-2.5} is completed.
\end{proof}

\end{document}